\documentclass[12pt]{amsart}

\usepackage{amsmath,url,graphicx,amssymb,enumerate,stmaryrd, amsthm}
\usepackage[pagebackref,colorlinks,citecolor=blue,linkcolor=blue,urlcolor=blue,filecolor=blue]{hyperref}
\usepackage{microtype}
\usepackage[all]{xy}
\usepackage{dsfont}
\usepackage{aliascnt}
\usepackage{mathrsfs}
\usepackage[vcentermath]{youngtab}
\usepackage{pinlabel}
\usepackage{etoolbox}
\usepackage{enumitem}
\usepackage{tikz} \usetikzlibrary{matrix,arrows}
\usepackage{pst-node}
\usepackage{tikz-cd} 
\usepackage[hmargin=4.5cm,vmargin=4cm]{geometry}

\numberwithin{thmcounter}{section}
\newaliascnt{thmauto}{thmcounter}

\newaliascnt{Defauto}{thmcounter}

\newaliascnt{exauto}{thmcounter}

\newaliascnt{quauto}{thmcounter}

\newaliascnt{lemauto}{thmcounter}

\newaliascnt{propauto}{thmcounter}

\newaliascnt{corauto}{thmcounter}

\newaliascnt{remauto}{thmcounter}

\newaliascnt{convauto}{thmcounter}

\newtheorem*{ThmA'}{Theorem A'}
\newtheorem*{ThmB'}{Theorem B'}
\newtheorem*{ThmC'}{Theorem C'}

\newtheorem{theorem}[thmauto]{Theorem}
\newtheorem{lemma}[lemauto]{Lemma}

\newtheorem{proposition}[propauto]{Proposition}

\theoremstyle{definition}

\newtheorem{remark}[remauto]{Remark}

\title{Counting geodesics on prime-order $k$-differentials}

\author{Juliet Aygun}
\address{Department of Mathematics, Cornell University, USA}
\email{ja742@cornell.edu}

\date{\today}

\begin{document}

\maketitle
\begin{abstract}
We determine weak asymptotics of counting functions on generic surfaces in a component of a stratum of $k$-differentials when $k$ is prime and genus is greater than $2$. In order to do so, we classify the $GL^+(2,\mathbb{R})$-orbit closure of holonomy covers of components and apply \cite[Theorem 2.12]{EMM15} generalized to translation surfaces. We show that the $GL^+(2,\mathbb{R})$-orbit closure of these holonomy covers is generically a component of a stratum of translation surfaces or a hyperelliptic locus therein. 
\end{abstract}

\section{Introduction}
Suppose $X$ is a compact Riemann surface of genus $g$. A $k$-differential $\xi$ on $X$ is a section of the $k$-th power of the canonical line bundle on $X$. Locally, $\xi$ is of the form $f(z)(dz)^k$ where $f(z)$ is a meromorphic function defined on a local coordinate $z$ on $X$. The bundle $\Omega^k\mathcal{M}_g$ parameterizes non-zero $k$-differentials on Riemann surfaces in $\mathcal{M}_g$. Denote the set of singularities of $\xi$ on $X$ by $\Sigma(\xi)$. Gauss-Bonnet theorem requires that the orders of each point in $\Sigma(\xi)$ sum up to $k(2g-2)$. Consider $\mu = (m_1,...,m_n)$ to be an integral partition of $k(2g-2)$ with each entry greater than $-k$ (i.e. no higher order poles). Let $\Omega^k\mathcal{M}_g(\mu)$ be the stratum of $\Omega^k\mathcal{M}_g$ whose $k$-differentials have singularities of orders corresponding to the entries of $\mu$. Positive entries of $\mu$ are the orders of the zeros and negative entries are the negative orders of the poles. When $k>2$, we do not allow regular marked points. If $\xi$ is not globally the $d$-th power of a $(k/d)$-differential, we call $\xi$ \textit{primitive}. Let $\Omega^k\mathcal{M}_g(\mu)^{\text{prim}}$ be the locus of primitive $k$-differentials in $\Omega^k\mathcal{M}_g(\mu)$. We single out the case $k=1$ by denoting the Hodge bundle of abelian differentials by $\mathcal{H}_g$ and a stratum by $\mathcal{H}_g(\mu)$. We will often use $\omega$ as notation for an abelian differential. The pair $(X, \xi)$ is called a $(1/k)$\textit{-translation surface}, and more exceptionally, a \textit{translation surface} when $k=1$ and a \textit{half-translation surface} when $k=2.$ Each $(1/k)$-translation surface $(X,\xi)$ is associated to a collection of polygons in $\mathbb{C}$ with sides identified by translation with possible rotation by $\frac{2\pi}{k}j$, $j \in \mathbb{Z}/k$, unique up to cut-and-paste. By pulling back the flat metric on $\mathbb{C}$, a $k$-differential induces a flat cone metric on $X$ with a cone angle of $2\pi(1 + \frac{m}{k})$ at a zero of order $m$ or $2\pi(1 - \frac{m}{k})$ at a pole of order $m$. The area of the flat metric on $X$ is denoted by $\text{Area}(X,\xi)$.

A \textit{cylinder core curve} on a $(1/k)$-translation surface is a closed geodesic disjoint from singularities. The union of cylinder core curves in the same isotopy class rel singularities forms what appears to be a `thickened geodesic,' or when $k \in \{1,2\}$, a Euclidean cylinder, hence the name `cylinder core curve.' For any $(1/k)$-translation surface, these `thickened geodesics' will be called \textit{cylinders}. Geodesics between two not necessarily distinct singularities which have no singularity in their interiors are called \textit{saddle connections}. A new and complicating feature for when $k>2$ is that cylinders and saddle connections often self-intersect because of non-trivial holonomy.

A natural question given a $(1/k)$-translation surfaces is how many cylinders or saddle connections of length less than $L$ does it have? The canonical length element of a curve $\gamma$ on a $(1/k)$-translation surface $(X,\xi)$ is the integral of some branch $|\sqrt[k]{\xi}|$ over $\gamma$. Functions which input a $(1/k)$-translation surface $M$ and length $L$ and output the number of cylinders or saddle connections of length less than $L$ are called \textit{counting functions}. Let $N_{cyl}(M,L)$ and $N_{sc}(M,L)$ denote these respective counting functions for $M \in \Omega^k\mathcal{M}_g(\mu)$. 

It is of popular interest to compute the asymptotics of $N_{cyl}$ and $N_{sc}$ on flat surfaces, long dating back to Masur in the 1980s and continued by Eskin, Mirzakhani, and Zorich in the 1990s and 2000s. Eskin-Masur \cite{EM01} found that for almost every translation surface, these exact asymptotics are $\pi L^2$ times a constant called a \textit{Siegel-Veech constant} (re-normalized by the area of the surface). The Siegel-Veech constants for generic translation surfaces can be computed using techniques in \cite{EMZ03}. Athreya-Eskin-Zorich \cite{AEZ16} computed Siegel-Veech constants for generic genus zero half-translation surfaces, and Goujard \cite{Gou15} extended their results to positive genus half-translation surfaces. In general, it is unknown if these exact asymptotics exist for surfaces when $k>2$. Following the notation of Athreya-Eskin-Zorich \cite{AEZ16}, we will take $N_*(M,L) `` \sim ” c · L^2$ to mean that $$\lim_{L \to \infty}\frac{1}{L}\int^{L}_{0} N_*(M,e^t)e^{-2t}dt = c.$$ These limits pertaining to Ces\`aro averages are referred to as \textit{weak asymptotics}. Weak asymptotics do exist for every $(1/k)$-translation surfaces as we will see below.

In this paper, we initiate the study of the asymptopics of counting functions on positive genus $(1/k)$-translation surfaces when $k>2$. We determine the weak asymptotics of generic surfaces when $k$ is prime and $g  >2$. Components of strata of prime-order $k$-differentials consist either of global $k$-th powers of abelian differentials or primitive $k$-differentials. In the former case, it is immediate that the asymptotics for any surface $(X,\xi)$ are the ones associated to the translation surface $(X,\sqrt[k]{\xi})$ (using any branch of $\sqrt[k]{\xi}$) by definition of the length element. Because we can already compute Siegel-Veech constants for generic translation surfaces in any stratum, we focus on the latter case.

Every $(1/k)$-translation surface ``unfolds'' to a canonical translation surface called a holonomy cover (see Section \ref{subsection_cover}). The measure on $\Omega^k\mathcal{M}_g(\mu)$ can be thought of as Lebesgue measure on local cohomological coordinates on the locus of holonomy covers (see \cite{Ngu22}). A surface $(X,\xi) \in \Omega^k\mathcal{M}_g$ is \textit{hyperelliptic} if $X$ is hyperelliptic and $\xi$ is a $(-1)^k$-eigenform of the hyperelliptic involution. In this paper, we also require the set of all (regular) marked points on $X$ to be invariant under the hyperelliptic involution. A connected component of a stratum is called a \textit{hyperelliptic component} if every $(1/k)$-translation surface inside is hyperelliptic.

\begin{theorem} \label{theorem_main}
Suppose $k>2$ is prime and $g>2$. Let $K$ be a component of $\Omega^k\mathcal{M}_{g}(\mu)^\emph{prim}$. There exists constants $\hat{c}_{cyl}(K)$ and $\hat{c}_{sc}(K)$ such that for almost every $M \in K$, $$N_{cyl}(M,L) `` \sim ” \frac{\hat{c}_{cyl} \cdot \pi L^2}{k^2 \cdot \emph{Area}(M)} \hspace{35pt} N_{sc}(M,L) `` \sim ” \frac{\hat{c}_{sc}\cdot \pi L^2}{k^2 \cdot \emph{Area}(M)}.$$ Let $\mathcal{N}$ be the locus of holonomy covers of surfaces in $K$, and let $\hat{K}$ be the connected component of $\mathcal{H}_{\hat{g}}(\hat{\mu})$ containing $\mathcal{N}$. Then $\hat{c}_{cyl}$ and $\hat{c}_{sc}$ are those Siegel-Veech constants associated to 
\begin{enumerate}
    \item $\hat{K}$ when $K$ is a non-hyperelliptic component or
    \item the hyperelliptic locus in $\hat{K}$ containing $\mathcal{N}$ when $K$ is a hyperelliptic component.
\end{enumerate} 
\end{theorem}

Recall Siegel-Veech constants for components of a stratum are computable using \cite{EMZ03}. One can also compute them for hyperelliptic loci therein using \cite{AEZ16} and \cite[Section 8]{Api21}. We discuss this in Section \ref{theorems}. The first part of Theorem \ref{theorem_main} follows quickly from Theorem \ref{weak} and Lemma \ref{lemma_NdenseinM}, so most of the work in this paper goes into finding $\hat{c}_{cyl}$ and $\hat{c}_{sc}$.

Veech \cite{Vee89} and Eskin-Marklof-Witte-Morris \cite{EMWM06} found exact asymptotics for counting functions of billiard trajectories on certain rational isosceles triangles (which correspond to cylinders on genus zero $(1/k)$-translation surfaces). Many followed and have found exact asymptotics for other types of rational billiard tables. Apisa \cite{Api21} found weak asymptotics for the remaining unknown cases of right and isosceles triangles. 

Translation surfaces and half-translation surfaces have a nice $GL^+(2,\mathbb{R})$-action which acts as linear transformations of their polygonal representations. The $GL^+(2,\mathbb{R})$-orbit closure of almost every translation surface is a component of its ambient stratum. Though the motivation for Theorem \ref{theorem_holonomyorbit} follows from Theorem \ref{weak}, it is independently an interesting result.

\begin{theorem} \label{theorem_holonomyorbit}
Suppose that $k>2$ is prime and $g>2$, and let $K$ be a component of $\Omega^k\mathcal{M}_g (\mu)^{\emph{prim}}$. Almost every $(X,\xi) \in K$ unfolds to a surface $(\hat{X},\hat{\omega}) \in \mathcal{H}_{\hat{g}}(\hat{\mu})$ whose $GL^+(2,\mathbb{R})$-orbit closure is
    \begin{enumerate}
        \item the ambient connected component of $\mathcal{H}_{\hat{g}}(\hat{\mu})$ when $K$ is non-hyperelliptic or
        \item a full hyperelliptic locus when $K$ is hyperelliptic. In particular, it is branched double covers of the stratum
        \begin{enumerate}
                \item $\Omega^2 \mathcal{M}_0(2m_1+k-2, 2m_2 + k -2, -1^{2gk})$ when $K$ is the hyperelliptic component of $\Omega^k \mathcal{M}_g(2m_1, 2m_2)$,
                \item $\Omega^2 \mathcal{M}_0(2m+k-2, 2\ell + 2k -2, -1^{2gk+k})$ when $K$ is the hyperelliptic component of $\Omega^k \mathcal{M}_g(2m, \ell, \ell)$,
                \item and $\Omega^2 \mathcal{M}_0(2\ell_1+2k-2, 2\ell_2 + 2k -2, -1^{2gk+2k})$ when $K$ is the hyperelliptic component of $\Omega^k \mathcal{M}_g(\ell_1, \ell_1, \ell_2, \ell_2)$.
            \end{enumerate} 
    \end{enumerate}
\end{theorem}
According to Chen-Gendron \cite{CG22}, hyperelliptic components of primitive $k$-differentials are classified as in ($a$), ($b$), or ($c$) of Theorem \ref{theorem_holonomyorbit}. We will also prove a similar result for many low genus components (Theorem \ref{theorem_LGC}). There is an extra possibility for the orbit closure of the holonomy covers of a stratum when $g\leq 2$, which is a non-arithmetic subvariety. We are unable to classify when this phenomenon occurs exactly, and hence cannot determine the asymptotics for all low genus components. In contrast to our result, many mathematicians beginning with Veech \cite{Vee89} have found non-arithmetic orbit closures of holonomy covers of genus zero strata. However, Mirzakhani-Wright \cite{MW18} also found infinitely many genus zero strata which unfold to surfaces with a dense $GL^+(2,\mathbb{R})$-orbit. Apisa \cite{Api21} classified the orbit closures of hyperelliptic holonomy covers of genus zero strata and obtained both low and high dimensional orbit closures. Outside of genus zero and $k \in \{1,2\}$, the orbit closures of holonomy covers have never been computed until now. 

Aside from when $k \in \{1,2\}$, not much is understood about strata of $(1/k)$-translation surfaces. This paper introduces different techniques and properties of strata which hopefully can be useful later. For instance, Lemma \ref{lemma_simplecylinder} proves the existence of Euclidean cylinders in many strata, and having such a cylinder to collapse is a useful tool for proofs by induction. Naturally, periodic trajectories on rational billiard tables and Platonic solids correspond to cylinders on $(1/k)$-translation surfaces (for instance, see \cite{Api21} and \cite{AAH22}). Holomorphic quadratic differentials correspond to the cotangent bundle of Teichm\"uller space. Higher-order differentials correspond to more abstract geometric structures. For instance, cubic differentials appear in the study of convex projective structures and quartic and sextic differentials in the study of Hitchin components. 

\subsection{Outline} In Section \ref{section_preliminaries}, we discuss known results and preliminaries pertaining to counting functions on translation surfaces, $(1/k)$-translation surfaces, and affine invariant subvarieties. In Section \ref{lemmas}, we prove Theorem \ref{theorem_holonomyorbit} and the partial result in low genus, Theorem \ref{theorem_LGC}. In Section \ref{theorems}, we use Theorem \ref{theorem_holonomyorbit} to obtain and discuss the weak asymptotics of counting functions on $(1/k)$-translation surfaces when $k>2$ is prime and $g>2$.

\subsection{Acknowledgments} The author is grateful for her advisor, Ben Dozier, for suggesting this problem and for his time and mentorship throughout the project. The author extends special thanks to Paul Apisa for many helpful conversations and John Rached for helpful conversations and comments on the draft. The author also thanks Dawei Chen, Samuel Grushevsky, and Alex Wright for helpful comments and conversations.

\section{Preliminaries} \label{section_preliminaries}

\subsection{Counting functions for translation surfaces}
By Eskin-Masur \cite{EM01}, almost every translation surface in a given component of a stratum has the same asymptotics for a counting function  (up to re-normalizing by the area).

\begin{theorem}[Eskin-Masur] \label{SV_const}
    For every connected component $K$ of $\mathcal{H}_g(\mu)$, there exists constants $c_{cyl}$ and $c_{sc}$ such that for almost every $M \in K$, the counting functions $N_{cyl}(M,L)$ and $N_{sc}(M,L)$ have the quadratic asymptotics $$\lim_{L\to \infty} \frac{N_{cyl}(M,L)}{\pi L^2} = \frac{c_{cyl}}{\emph{Area}(M)} \hspace{35pt} \lim_{L\to \infty} \frac{N_{sc}(M,L)}{\pi L^2} = \frac{c_{sc}}{\emph{Area}(M)}.$$ 
\end{theorem} 
The constants $c_{cyl}$ and $c_{sc}$ are called \textit{Siegel-Veech constants}. In \cite{EMZ03}, Siegel-Veech constants associated to a component $K$ of $\mathcal{H}_g(\mu)$ were computed in terms of volumes of unit area hyperboloids in boundary strata adjacent to the cusp in $K$ where the length of the configuration is short. In \cite{AEZ16} and \cite{Gou15}, these ideas were generalized to strata of half-translation surfaces. Eskin-Okounkov \cite{EO01} computed volumes of unit area hyperboloids of strata of translation surfaces and Goujard \cite{Gou16} of half-translation surfaces. These volumes are obtained by coning off the hyperboloids and then taking its Masur-Veech volume. The measure zero set in a component excluded from Theorem \ref{SV_const} is not well-understood. 

Eskin-Mirzakhani-Mohammadi \cite{EMM15} proved that the weak asymptotics of counting functions for translation surfaces only depend on their $GL^+(2,\mathbb{R})$-orbit closures. It is conjectured that the full measure set in Theorem \ref{SV_const} includes all surfaces with a dense $GL^+(2,\mathbb{R})$-orbit, and the extra averaging in the following Theorem is unnecessary. 

\begin{theorem}[Eskin-Mirzakhani-Mohammadi] 
\label{weak} 
For any $M \in \mathcal{H}_g(\mu)$, there are constants $c,s>0$ dependent on $\overline{GL^+(2,\mathbb{R})M}$ such that 
$$N_{cyl}(M,L) `` \sim ”  \frac{c \cdot \pi L^2}{\emph{Area}(M)} \hspace{35pt} N_{sc}(M,L) `` \sim ” \frac{s \cdot \pi L^2}{\emph{Area}(M)}.$$ 
\end{theorem}

This associates to $M$ an average of $N_{cyl}(M,L)$ and $N_{sc}(M,L)$ as $L \to \infty$. Moreover, recall almost every surface in a component $K$ in $\mathcal{H}_g(\mu)$ has a dense $GL^+(2,\mathbb{R})$-orbit in $K$. One can then take $c$ and $s$ in the Theorem above to be the Siegel-Veech constant associated to $K$ for these generic surfaces.

Therefore, if the holonomy cover, defined below, $(\hat{X},\hat{\omega})$ of $(X,\xi)$ has a dense $GL^+(2,\mathbb{R})$-orbit in a component of a stratum, we know the weak asymptotics for $(\hat{X},\hat{\omega})$. There is a simple relationship between the counting functions on $(X,\xi)$ and the counting functions on $(\hat{X},\hat{\omega})$, so proving Theorem \ref{theorem_holonomyorbit} is the main component of this paper.

\subsection{Holonomy covers} \label{subsection_cover}
We can always unfold a $(1/k)$-translation surface into a translation surface, formally called its \textit{holonomy cover}. More precisely, given $(X,\xi) \in \Omega^k\mathcal{M}_g(\mu)$, it is a canonical ramified cyclic cover $\pi: (\hat{X},\hat{\omega}) \to (X,\xi)$ of degree $k$ such that the pullback of $\xi$ is the $k$-th power of the abelian differential $\hat{\omega}$ on $\hat{X}$. The holonomy cover is only branched over zeros and poles of $\xi$, and $\hat{X}$ is connected if and only if $\xi$ is primitive. Furthermore, there is a generator $\tau$ of the cyclic deck group of $\hat{X}$ associated to a primitive $k$-th root of unity $\zeta$ such that $\tau^*\hat{\omega} = \zeta \hat{\omega}$. The generator $\tau$ induces another periodic automorphism on $H_1(\hat{X},\Sigma(\hat{\omega});\mathbb{C})$ and $H^1(\hat{X},\Sigma(\hat{\omega});\mathbb{C})$, thus decomposing them into respective eigenspaces $H_1(\hat{X},\Sigma(\hat{\omega});\mathbb{C})_1,...,H_1(\hat{X},\Sigma(\hat{\omega});\mathbb{C})_{\zeta^{k-1}}$ and $H^1(\hat{X},\Sigma(\hat{\omega});\mathbb{C})_1,...,H^1(\hat{X},\Sigma(\hat{\omega});\mathbb{C})_{\zeta^{k-1}}$ associated to the eigenvalues $1,\zeta,...,\zeta^{k-1}.$  An intuitive way to think about the construction of the holonomy cover is to take $k$ copies of the polygonal representation of $(X,\xi)$, rotate each one by a different multiple of $\frac{2\pi}{k}$, and then re-label sides so pairs are identified by translation. See Figures \ref{fig: cover_construction} and \ref{fig: geodesiccover}. \begin{figure}
    \begin{center}
    {\includegraphics[width=0.65\textwidth]{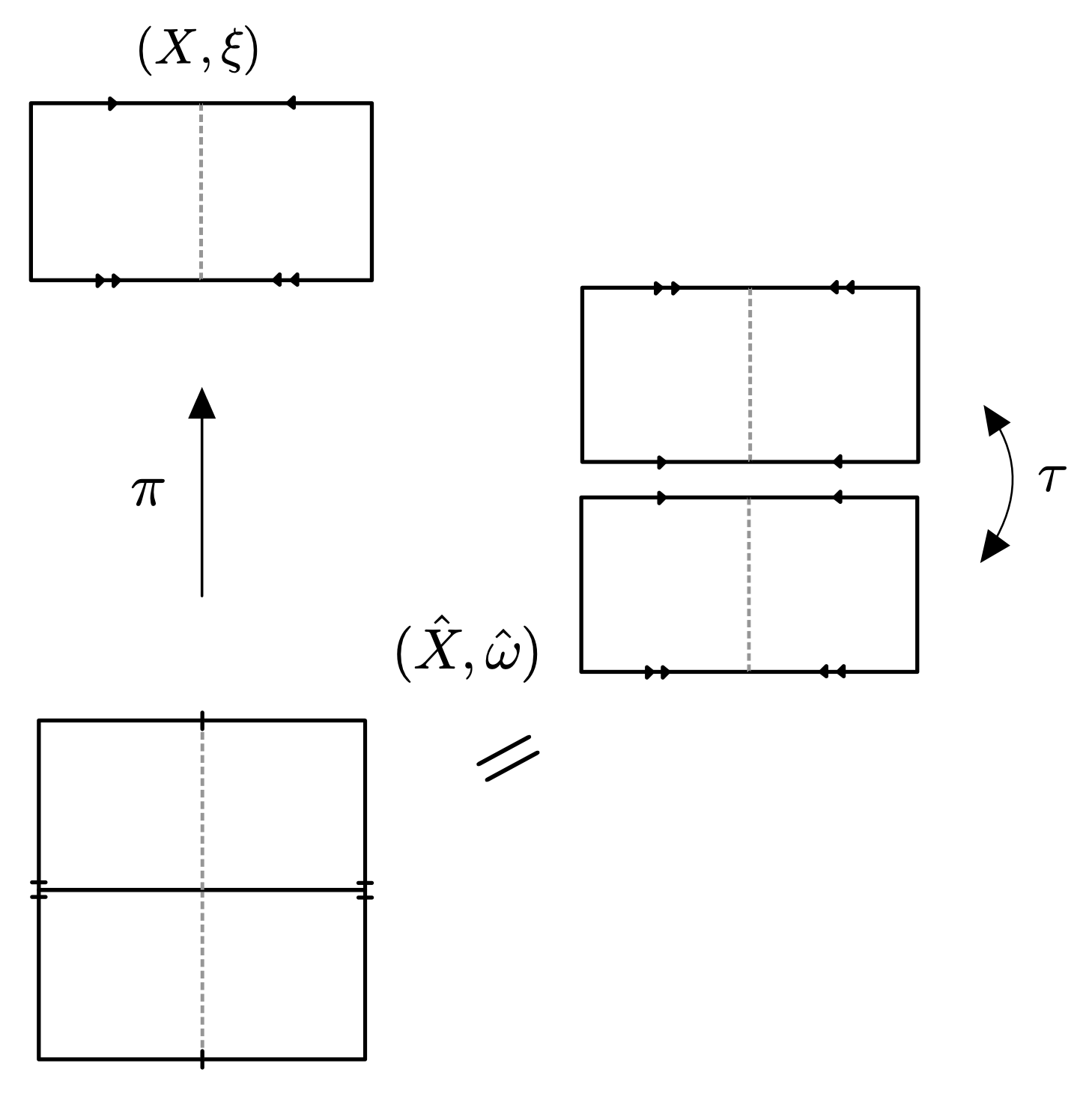}}
    \end{center}
    \caption{Holonomy cover construction of a genus zero quadratic differential. This particular example is called the pillowcase surface.}
    \label{fig: cover_construction}
\end{figure} 

We let $\mathcal{H}_{\hat{g}}(\hat{\mu})$ be the ambient stratum of the holonomy cover $(\hat{X},\hat{\omega})$. Its partition $\hat{\mu}$ may have $0$ as some entries if the pre-image of some poles is a (regular) marked point on $\hat{X}$. The unusal convention is to regard these marked points as zeros of order zero. The component $\mathcal{H}_{\hat{g}}(\hat{\mu})$ is then a point-marking which fibers over the underlying surfaces without marked points. The following Proposition is from \cite[Proposition 2.4]{BCGGM19}.

\begin{proposition} [Riemann-Hurwitz Formula] The ambient stratum $\mathcal{H}_{\hat{g}}(\hat{\mu})$ of the holonomy cover $(\hat{X},\hat{\omega})$ of $(X,\xi) \in \Omega^k\mathcal{M}_g(\mu)^{\emph{prim}}$ has the following properties given the partition $\mu = (m_1,...,m_n)$.
\begin{enumerate}
    \item The genus $\hat{g}$ of $\hat{X}$ is $$\hat{g} = 1+k(g-1) +\frac{1}{2}\left(kn-\sum_{j=1}^{n}\emph{gcd}(m_j,k)\right).$$ 
    \item The partition $\hat{\mu}$ of $\mathcal{H}_{\hat{g}}(\hat{\mu})$ is $$\hat{\mu} = (\underbrace{\hat{m}_1,...,\hat{m}_1}_{\emph{gcd}(m_1,k)}, \underbrace{\hat{m}_2,...,\hat{m}_2}_{\emph{gcd}(m_2,k)},...,\underbrace{\hat{m}_n,...,\hat{m}_n}_{\emph{gcd}(m_n,k)})$$ where $\hat{m}_j:=\frac{m_j+k}{\emph{gcd}(m_j,k)}-1.$
\end{enumerate}
\end{proposition}

\begin{remark} \label{remark_markedpt}
  When $k$ is prime, marked points on $(\hat{X},\hat{\omega})$ are always the pre-images of poles of order $k-1$ on $(X,\xi)$.
\end{remark}

\begin{remark}
The locus of holonomy covers of a genus zero stratum of half-translation surfaces is a hyperelliptic locus in which we additionally mark the pre-images of poles.
\end{remark}

\subsection{Geodesics and local coordinates} 
For any translation surface $(S,\omega)$, the \textit{period} of any oriented path $\gamma$ on $S$, defined by $\text{hol}(\gamma) : = \int_{\gamma} \omega$, is well-defined. This means that all cylinder core curves and saddle connections on $(S,\omega)$ have a constant, well-defined slope. When $k > 1$, however, direction is only well-defined up to an angle of $\frac{2\pi}{k}$. 

Away from singularities, $\pi: (\hat{X},\hat{\omega}) \to (X,\xi)$ is a Riemannian cover. Therefore, the lift of a cylinder (or saddle connection) on $(X,\xi)$ is a union of cylinders (resp. saddle connections) on $(\hat{X},\hat{\omega})$, and the periods of the lifts all differ by multiplication by a $k$-th root of unity. In particular, there are $k$ geodesics in the pre-image of a cylinder or saddle connection. Moreover, all cylinders (resp. saddle connections) project to cylinders (resp. saddle connections) on $(X,\xi)$. \begin{figure}
    \begin{center}
    {\includegraphics[width=.9\textwidth]{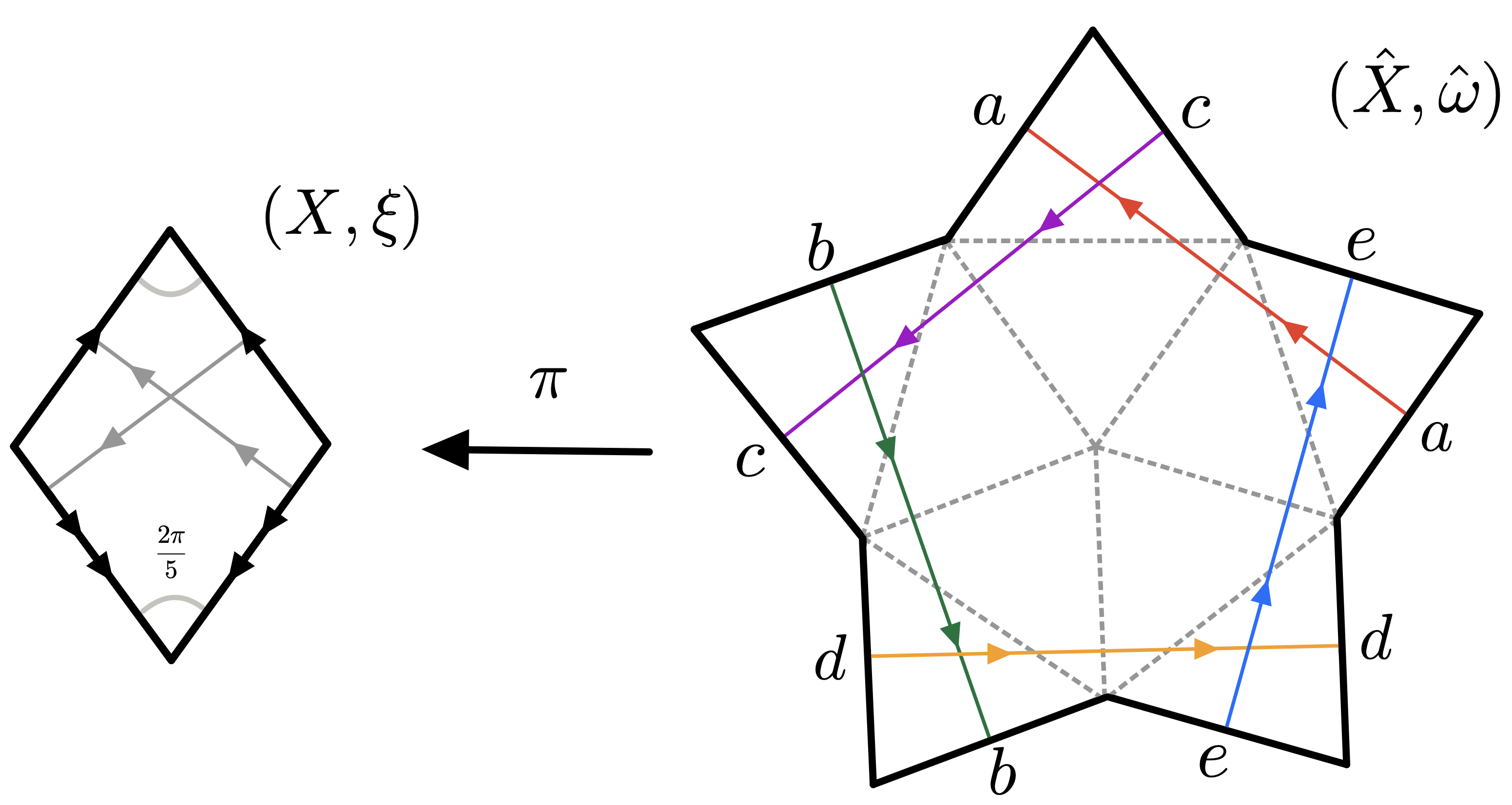}}
    \end{center}
    \caption{The lift of a self-intersecting cylinder core curve on a $(1/5)$-translation surface $(X,\xi)$ to its holonomy cover $(\hat{X},\hat{\omega})$. The lift is a union of five cylinder core curves. Edges labeled with the same letter get identified.}
    \label{fig: geodesiccover}
\end{figure} Thus, if $M = (X,\xi)$ and $\hat{M} = (\hat{X},\hat{\omega})$, there are the following counting relations \begin{equation} \label{equation_count}  N_{cyl}(\hat{M},L) = k \cdot N_{cyl}(M,L)  \hspace{35pt} N_{sc}(\hat{M},L) = k \cdot N_{sc}(M,L).  \end{equation} We emphasize these relations only hold because we mark the pre-images of poles which are regular (Remark \ref{remark_markedpt}). Additionally, it is obvious that \begin{equation} \label{equation_area}  \text{Area}(\hat{M}) = k \cdot \text{Area}(M) . \end{equation}

The period of some branch of the $k$-th root $\sqrt[k]{\xi}$ over a curve $\gamma$ is equal to the period of $\hat{\omega}$ over $\frac{1}{k}\sum_{j=0}^{k-1} \bar{\zeta}^{j}\tau_*^j([\hat{\gamma}])$ where $\hat{\gamma}$ is a lift of $\gamma$ to $(\hat{X},\hat{\omega})$ chosen depending on the branch of $\sqrt[k]{\xi}$. We say two cylinder core curves $\gamma$ and $\gamma'$ on $(X,\xi)$ are \textit{hat homologous} if $$\sum_{j=1}^k \bar{\zeta}^{j}\tau^j_*([\hat{\gamma}_i]) = r\sum_{j=1}^k \bar{\zeta}^{j}\tau^j_*([\hat{\gamma}'_i])$$ for some $r \in \mathbb{R}$ and choice of lifts $\hat{\gamma}_i$ and $\hat{\gamma}_i'$ for $\gamma$ and $\gamma'$ respectively. We can take $r$ to be in $\mathbb{Q}(\zeta) \cap \mathbb{R}$ because these hat homology classes are defined over $\mathbb{Q}(\zeta)$. Because these periods differ by a real number, $\hat{\gamma}_i$ and $\hat{\gamma}_i'$ are parallel locally in the locus of holonomy cover $\mathcal{N}$.

Surfaces in the primitive locus of a stratum are locally determined by periods of curves on their holonomy cover \cite[Corollary 2.3]{BCGGM19} and the tangent space of the locus of holonomy covers is given by the $\zeta$-eigenspace of relative cohomology \cite[Theorem 2.2]{BCGGM19}. We state these results below.

\begin{theorem} \label{theorem_localcoord} \emph{(Bainbridge-Chen-Gendron-Grushevsky-M\"oller)}
   Locally at $(X,\xi)$, the locus $\Omega^k\mathcal{M}_g (\mu)^{\emph{prim}}$ has local coordinates given by the periods $\emph{hol}(\beta_i)$ on the holonomy cover $(\hat{X},\hat{\omega})$ where $\{\beta_i\}$ is a basis for $H_1(\hat{X}, \Sigma(\hat{\omega});\mathbb{C})_{\zeta}$. Moreover, the tangent space of $\mathcal{N}$ at $(\hat{X},\hat{\omega})$ is $H^1(\hat{X}, \Sigma(\hat{\omega});\mathbb{C})_{\zeta}$.
\end{theorem}

We refer to these coordinates as \textit{cohomological coordinates}. Without resorting to the cover, it is not hard to see that the periods of some curves with respect to some branches of $\sqrt[k]{\xi}$ locally determine $(X,\xi)$. However, this system will not be useful to us. Accordingly, components of strata are orbifolds. Below is \cite[Theorem 1.1]{BCGGM19}.

\begin{theorem} \label{theorem_dimofstrata} \emph{(Bainbridge-Chen-Gendron-Grushevsky-M\"oller)}
   Every connected component of the stratum $\Omega^k\mathcal{M}_g (\mu)$ is a smooth orbifold. If the component parameterizes $k$-th powers of holomorphic abelian differentials, it has complex dimension $2g +n- 1$. Otherwise, it has complex dimension $2g +n -2 $.
\end{theorem}

The one dimension missing from Theorem \ref{theorem_dimofstrata} when a component does not parameterize $k$-th powers of abelian differentials is a consequence of non-trivial holonomy.

\subsection{Affine invariant subvarieties of $\mathcal{H}_g$} 
Eskin-Mirzakhani-Mohammadi \cite{EMM15} and Filip \cite{Fil16} together show that the $GL^+(2,\mathbb{R})$-orbit closure of any surface is an \textit{affine invariant subvariety} of $\mathcal{H}_g(\mu)$ i.e. an immersed subvariety whose image is locally defined by linear equations with real coefficients in cohomological coordinates (see \cite{Wri14}). When $k>2$, a locus $\mathcal{N}$ of holonomy covers of a component of $\Omega^k\mathcal{M}_g(\mu)^\text{prim}$ is never an affine invariant subvariety, but rather a complex linear subvariety i.e. a subvariety locally defined by linear equations with complex coefficients in cohomological coordinates. Indeed, some elements of $GL^+(2,\mathbb{R})$ move surfaces outside the locus.

Wright \cite{Wri15b} proved that any translation surface in $\mathcal{H}_g(\mu)$ whose periods in cohomological coordinates are linearly independent over $\bar{\mathbb{Q}} \cap \mathbb{R}$ has a dense orbit in a component of $\mathcal{H}_g(\mu)$. One might hope the collection $\mathcal{N}$ of holonomy covers of $(1/k)$-translation surfaces for certain $k$ is never locally contained in a non-trivial linear subspace with coefficients in $(\bar{\mathbb{Q}} \cap \mathbb{R})$ in cohomological coordinates. However, for any $k>1$, one can scale and add the linear equations cutting out $\mathcal{N}$ together to form a linear equation with coefficients in $(\bar{\mathbb{Q}} \cap \mathbb{R})$. Hence, more advanced techniques are needed to prove Theorem \ref{theorem_holonomyorbit}.

\subsection{Cylinder deformations} \label{subsection_cyldef}
In this subsection, we briefly discuss the tangent space of affine invariant subvarieties in $\mathcal{H}_g$. Let $\mathcal{M}'$ be any affine invariant subvariety clear from context, and let $H^1_{rel}$ be the natural bundle over $\mathcal{M}'$ whose fiber over a translation surface $(S,\omega)$ is $H^1(S, \Sigma(\omega);\mathbb{C})$. The tangent space $T\mathcal{M}'$ of $\mathcal{M}'$ is a flat subbundle of $H^1_{rel}$. Let $p:H^1(S,\Sigma(\omega);\mathbb{C}) \to H^1(S;\mathbb{C})$ be the natural projection map to absolute cohomology. Define the \textit{rank} of $\mathcal{M}'$ to be $$\text{rank}(\mathcal{M}') := \frac{1}{2}\text{dim}_{\mathbb{C}}p(T_{(S,\omega)}\mathcal{M}')$$ for any surface $(S,\omega) \in \mathcal{M}'$. In particular, $\mathcal{M}'$ has rank $g$ if and only if locally the absolute periods of any $(S,\omega) \in \mathcal{M}'$ are unconstrained.

Apisa-Wright \cite{AW23} proved that an affine invariant subvariety with sufficient rank in a component of a stratum must be the component or holonomy double covers of a stratum of quadratic differentials (after forgetting marked points). 

\begin{theorem}[Apisa-Wright] \label{theorem_highrank}
Let $\mathcal{M}'$ be an affine invariant subvariety without marked points in a stratum $\mathcal{H}_g(\mu)$ with $\emph{rank}(\mathcal{M}') \geq \frac{g}{2} + 1$. Then $\mathcal{M}'$ is either a connected
component of a stratum or the locus of holonomy covers of surfaces in a stratum of half-translation surfaces with forgotten marked points.
\end{theorem}

When an affine invariant subvariety (with possibly marked points) satisfies the rank assumption of Theorem \ref{theorem_highrank}, it is called \textit{high rank}. If its rank is equal to $g$, then it is called \textit{full rank}. Mirzakhani-Wright \cite{MW18} proved that full rank affine invariant subvarieties are trivial.

\begin{theorem}[Mirzakhani-Wright] \label{theorem_fullrank}
Let $\mathcal{M}'$ be a full rank affine invariant subvariety without marked points. Then $\mathcal{M}'$ is either a connected component of a stratum, or a hyperelliptic locus therein.
\end{theorem}

Recall that the one-parameter subgroup \[
u_t=
  \begin{bmatrix}
    1 & t  \\
    0 & 1   \end{bmatrix}
 \subset GL^+(2, \mathbb{R})\] shears polygons in the plane. Let $\mathcal{C}$ be a collection of parallel cylinders on $(S,\omega)$ pointing in direction $\theta$. We define the \textit{cylinder shear} $u^\mathcal{C}_t(S,\omega)$ to be the surface obtained by rotating $(S,\omega)$ by $-\theta$ so that the cylinders in $\mathcal{C}$ are pointing in the positive horizontal direction, applying $u_t$ to the cylinders in $\mathcal{C}$, and rotating the resulting surface back by $\theta.$ 

 Recall the Poinc\'are isomorphism $$H_1(S-\Sigma(\omega);\mathbb{C}) \cong H^1(S,\Sigma(\omega);\mathbb{C})$$ which is given by the intersection number. If $\alpha$ is a closed curve on $S$, let $I(\alpha) \in H^1(S, \Sigma(\omega);\mathbb{Z})$ denote the dual of its class in $H_1(S-\Sigma(\omega);\mathbb{Z})$. If $\alpha_j$ and $h_j$ are a core curve and the height of a cylinder $C_j \in \mathcal{C}$, the derivative of $u^\mathcal{C}_t$ at $(S,\omega)$ is $u^\mathcal{C} := e^{i\theta}\sum_{j=1}^mh_jI(\alpha_j).$ Similarly, there is a deformation of $(S,\omega)$ called a \textit{cylinder stretch} whose derivative is $i u^\mathcal{C}$ which stretches, rather than shears, the cylinders in $\mathcal{C}$.

The Cylinder Deformation Theorem \cite{Wri15a} stated below describes deformations of surfaces remaining in $\mathcal{M}'$ obtained by shearing and stretching cylinders. Given any (not necessarily affine invariant) subvariety $\mathcal{M}'$, a collection of cylinders $\mathcal{C}$ on $(S,\omega) \in \mathcal{M}'$ are said to be $\mathcal{M}'$-\textit{parallel}, or in the same $\mathcal{M}'$-parallel equivalence class, if their core curves remain parallel on a sufficiently small neighborhood of $(S,\omega)$ in $\mathcal{M}'$. 

\begin{theorem}[Wright] \label{theorem_cyldefthm}
    \emph{(The Cylinder Deformation Theorem)} Let $\mathcal{M}'$ be an affine invariant subvariety containing $(S,\omega)$, and let $\mathcal{C}$ be a full equivalence class of $\mathcal{M}'$-parallel cylinders on $(S,\omega)$. Then, for all sufficiently small $t \in \mathbb{R}$, the surface $u^{\mathcal{C}}_t(S,\omega)$ remains in $\mathcal{M}'$. In particular, if $C_j$ in $\mathcal{C}$ has core curve $\alpha_j$ and height $h_j$, then $\lambda \sum_{j=1}^mh_jI(\alpha_j) \in T_{(S,\omega)}\mathcal{M}'$ for any $\lambda \in \mathbb{C}$. 
\end{theorem}

For more information on cylinder deformations and the tangent bundle of an affine invariant subvariety in $\mathcal{H}_g$, see \cite{Wri15a}.

\section{Proof of Theorem \ref{theorem_holonomyorbit}} \label{lemmas}

Unless otherwise stated, we will always let $\mathcal{N}$ be a locus of holonomy covers of a component of $\Omega^k\mathcal{M}_g(\mu)^{\text{prim}}$ and $(\hat{X},\hat{\omega}) \in \mathcal{N}$ the holonomy cover of $(X,\xi)$ with $k$-cyclic automorphism $\tau$. We fix a choice of a primitive $k$-th root of unity $\zeta$. Let $\mathcal{M}$ be the smallest affine invariant subvariety containing $\mathcal{N}$.

One can think of $\mathcal{M}$ as the orbit closure of a generic surface in $\mathcal{N}$, as shown below.

\begin{lemma} \label{lemma_NdenseinM}
    Almost every surface in $\mathcal{N}$ has a dense $GL^+(2,\mathbb{R})$-orbit in $\mathcal{M}$.
\end{lemma}
\begin{proof}
Recall $\mathcal{M}$ is an affine invariant subvariety, and there are only countably many proper affine invariant subvarieties contained in $\mathcal{M}$ by \cite{EMM15}. Recall also that $\mathcal{N}$ is a complex linear subvariety. Hence, these countably many proper subvarieties in $\mathcal{M}$ intersect $\mathcal{N}$ at a measure zero subset (with respect to Lebesgue measure on $\mathcal{N}$) or $\mathcal{N}$ is contained in this countable union. If it was the latter, $\mathcal{M}$ is not the smallest affine invariant subvariety containing $\mathcal{N}$ which is a contradiction.

Therefore, the $GL^+(2,\mathbb{R})$-orbit closure of any surface aside from a measure zero set in $\mathcal{N}$ cannot be a proper subvariety in $\mathcal{M}$ and thus is $\mathcal{M}$. 
\end{proof}

The following is Lemma 4.1 and a consequence of Proposition 4.2 in \cite{Ngu22}.

\begin{lemma}[Nguyen] \label{lemma_projection}
    We have the equality $$p(H^1(\hat{X},\Sigma(\hat{\omega});\mathbb{C})_{\zeta}) = H^1(\hat{X};\mathbb{C})_\zeta.$$ Moreover, set $$N:= 2g+n-2 - \emph{card}\{m_1,...,m_n \cap k\mathbb{Z}\}.$$ Then $\emph{dim}(H^1(\hat{X};\mathbb{C})_{\zeta}) = N$ when $k>1.$
\end{lemma} 

\begin{remark}
If $\alpha$ is a closed curve on a surface $(S,\omega)$, we may consider it as a class in $H_1(S;\mathbb{C})$ and will denote its dual (under Poinc\'are duality in absolute homology) as $\alpha^* \in H^1(S;\mathbb{C})$.
In fact, one can consider $\alpha$ to be a class in both $H_1(S - \Sigma(\omega);\mathbb{C})$ and $H_1(S;\mathbb{C})$, and without changing notation for one or the other, it follows that $p(I(\alpha)) = \alpha^*$. 
\end{remark}

\begin{remark}
Given the automorphism $\tau$ on $(\hat{X},\hat{\omega})$, we also slightly abuse the notation of $\tau_*$ (or $\tau^*$) by allowing it to act as an induced action on both absolute and relative (co-)homology classes. Let $H_1(\hat{X};\mathbb{C})_{\zeta^\ell}$ (or $H^1(\hat{X};\mathbb{C})_{\zeta^\ell}$) denote the $\zeta^\ell$-eigenspace of $\tau_*$ (or $\tau^*$) acting on the absolute (co-)homology of $\hat{X}$.
\end{remark}

\begin{lemma} \label{lemma_zetaeigenspace}
    At any $(\hat{X},\hat{\omega}) \in \mathcal{N}$, $p(T_{(\hat{X},\hat{\omega})}\mathcal{M})$ contains the eigenspaces $H^1(\hat{X};\mathbb{C})_{\zeta}$ and $H^1(\hat{X};\mathbb{C})_{\bar{\zeta}}$. 
\end{lemma}
\begin{proof}
Recall in Theorem \ref{theorem_localcoord} that $H^1(\hat{X},\Sigma(\hat{\omega});\mathbb{C})_{\zeta}$ is the tangent space of $\mathcal{N}$ at $(\hat{X},\hat{\omega})$. Because $\mathcal{M}$ contains $\mathcal{N}$, $T_{(\hat{X},\hat{\omega})}\mathcal{M}$ contains $H^1(\hat{X},\Sigma(\hat{\omega});\mathbb{C})_{\zeta}$. Lemma \ref{lemma_projection} then implies $H^1(\hat{X};\mathbb{C})_{\zeta} \subset p(T_{(\hat{X},\hat{\omega})}\mathcal{M}).$ Recall $\mathcal{M}$ is defined locally by linear equations with real coefficients. Hence, $p(T_{(\hat{X},\hat{\omega})}\mathcal{M})$ has a basis in $H^1(\hat{X};\mathbb{R})$, so any element in $p(T_{(\hat{X},\hat{\omega})}\mathcal{M})$ has its conjugate in $p(T_{(\hat{X},\hat{\omega})}\mathcal{M}).$ If $\tau^*v = \bar{\zeta} v$, then $\tau^*\bar{v} = \zeta \bar{v}$ because $\tau^*$ is a real (integral) operator. Thus, all vectors in $H^1(\hat{X};\mathbb{C})_{\bar{\zeta}}$ can be obtained by conjugating some element in $H^1(\hat{X};\mathbb{C})_{\zeta}$. We conclude $H^1(\hat{X};\mathbb{C})_{\bar{\zeta}} \subset p(T_{(\hat{X},\hat{\omega})}\mathcal{M})$.
\end{proof}

\begin{lemma} \label{lemma_parallelsaddle}
On almost every $(\hat{X},\hat{\omega}) \in \mathcal{N}$, any two parallel cylinder core curves project to hat homologous curves on $(X,\xi)$. 
\end{lemma}
\begin{proof}
The period of a core curve $\hat{\gamma} \subset \hat{X}$ is also the period of $\frac{1}{k} \sum_{j=1}^k \bar{\zeta}^{j}\tau^j_*([\hat{\gamma}]) \in H_{1}(\hat{X}, \Sigma(\hat{\omega});\mathbb{C})_{\zeta}$. Recall that local coordinates for $\mathcal{N}$ are given by the periods of some basis in $H_1(\hat{X}, \Sigma(\hat{\omega});\mathbb{C})_{\zeta}$. If two core curves $\hat{\gamma}$ and $\hat{\gamma}'$ are not hat homologous downstairs, i.e. do not satisfy $$\frac{1}{k}\sum_{j=1}^k \bar{\zeta}^{j}\tau^j_*([\hat{\gamma}]) = \frac{r}{k} \sum_{j=1}^k \bar{\zeta}^{j}\tau^j_*([\hat{\gamma}'])$$ for some $r \in \mathbb{Q}(\zeta) \cap \mathbb{R}$, then the locus in which their periods all have real ratios is of real codimension one in $\mathcal{N}$. Since there are always countably many cylinders on any given translation surface, there are only countably many real codimension one loci where there are at least two parallel cylinder core curves which are not hat homologous downstairs. 
\end{proof}

Consider a stratum of translation surfaces possibly with marked points $\mathcal{H}_g(\mu)$, and let $\mathcal{F}:\mathcal{H}_g(\mu) \to  \mathcal{H}_g(\mu')$ be the map which forgets the marked points. Given a surface or subvariety, we call it \textit{unmarked} after applying $\mathcal{F}$. Forgetting marked points does not change the image of the tangent space under $p$ nor consequently the rank. However, it does allow us to discuss orbit closures without worrying about constraints on marked points. Understanding the orbit closure of an unmarked surface helps us understand how freely marked points may move around inside the closure. For that reason, we will classify $\mathcal{F(M)}$ first and then deal with marked points afterwards. We emphasize that we use $\mathcal{F(M)}$ at times rather than $\mathcal{M}$ as a precaution; certain papers we cite do not explicitly involve marked points in their context.

The remainder of this section splits into two parts: showing that $\mathcal{F}(\mathcal{M})$ is arithmetic (Section \ref{subsection_arithmetic}) and then classifying $\mathcal{M}$ (Section \ref{subsection_classification}). Lemma \ref{lemma_eigenspacefullrank} highlights why the second part relies greatly on first showing $\mathcal{F}(\mathcal{M})$ is arithmetic. From there, determining $\mathcal{M}$ is high rank is quick (Lemma \ref{lemma_Mhighrank}), which consequently constrains $\mathcal{M}$ to be a whole component or double covers of a stratum of quadratic differentials (ignoring the marked points). Later in Section \ref{subsection_classification}, we show the latter can only happen when the component of $(1/k)$-translation surfaces is hyperelliptic, and $\mathcal{M}$ in this case is a hyperelliptic locus itself. 

\subsection{Arithmeticity of $\mathcal{F}(\mathcal{M})$} \label{subsection_arithmetic}
This subsection takes the first, and perhaps most significant, step to prove Theorem \ref{theorem_holonomyorbit}: showing that $\mathcal{F(M)}$ is arithmetic. The \textit{field of definition} $\textbf{k}(\mathcal{M}')$ of an affine invariant subvariety $\mathcal{M}'$ is the smallest subfield of $\mathbb{R}$ for which $\mathcal{M}'$ can be locally defined by linear equations (in cohomological coordinates) with coefficients in this field. An affine invariant subvariety $\mathcal{M}'$ is \textit{arithmetic} if $\textbf{k}(\mathcal{M}') = \mathbb{Q}$. We will show a Euclidean cylinder, i.e. a cylinder with simple core curves, on any surface in $K$ implies $\mathcal{F(M)}$ is arithmetic. Then, we prove the existence of a Euclidean cylinder on some surface in $K$ when $g>2$.

Let $\hat{\iota}(\_,\_)$ denote the intersection form between two absolute homology classes and/or closed curve representatives on a surface. 

\begin{lemma} \label{lemma_sameintersectionsimple}
    Suppose that $k$ is prime and $\hat{C}$ is a cylinder on $(\hat{X},\hat{\omega})$ such that $\pi(\hat{C})$ is a Euclidean cylinder. Then, the collection of all core curves of cylinders in the $\mathcal{N}$-parallel equivalence class of $\hat{C}$ and the cylinders in their $\tau$-orbits are pairwise disjoint.
\end{lemma}

\begin{proof}
Suppose that $\hat{\gamma}$ is a core curve of $\hat{C}$. Because $\pi$ is a local homeomorphism away from singularities, every intersection between $\hat{\gamma}, \tau(\hat{\gamma}),...,\tau^{k-1}(\hat{\gamma})$ projects to a self-intersection of $\pi(\hat{\gamma})$ on $(X,\xi)$. If $\pi(\hat{\gamma})$ is simple, then $\hat{\gamma}, \tau(\hat{\gamma}),...,\tau^{k-1}(\hat{\gamma})$ are pairwise disjoint, i.e. $\hat{\iota}(\hat{\gamma}, \tau^j(\hat{\gamma})) = 0$ for all $j \in \{0,...,k-1\}$. Hence, $\hat{\gamma}$ pairs trivially with $$v := \sum_{j=0}^{k-1}\bar{\zeta}^{j}(\tau^*)^{j}(\hat{\gamma}^*) \in H^1(\hat{X};\mathbb{C})_{\zeta} =p(T_{(\hat{X},\hat{\omega})}\mathcal{N}).$$ Because the ratios of periods of curves $\mathcal{N}$-parallel to $\hat{\gamma}$ are rigid in a neighborhood of $(\hat{X},\hat{\omega})$ in $\mathcal{N}$, they too must pair trivially with $v$, and for that matter, $rv$ for any $r\in \mathbb{C}$.

Suppose there were two (possibly not distinct) core curves $\hat{\gamma}'$ and $\hat{\gamma}''$ that are $\mathcal{N}$-parallel to $\hat{\gamma}$ and such that for some $j,\ell \in \{0,...,k-1\}$, $$\hat{\iota}(\tau^j(\hat{\gamma}'),\tau^{\ell}(\hat{\gamma}'')) \neq 0.$$ We include the case $\hat{\gamma} = \hat{\gamma}'$. Because $\tau^{k-j}$ is an automorphism, the intersection form is $(\tau^{k-j})$-invariant and $$\hat{\iota}(\hat{\gamma}',\tau^{\ell+ (k-j)}(\hat{\gamma}'')) \neq 0.$$ Recall $k-1$ of the $k$-th roots of unity are rationally independent when $k$ is prime and $\hat{\iota}(\hat{\gamma}',\hat{\gamma}'') = 0$, so $\hat{\gamma}'$ pairs non-trivially with $\sum_{j=0}^{k-1}\bar{\zeta}^{j}(\tau^*)^{j}(\hat{\gamma}''^*)$. Using Poinc\'are duality on the relation between $\hat{\gamma}$ and $\hat{\gamma}''$ in (absolute) homology gives the relation $$\frac{1}{k}\sum_{j=1}^k \zeta^{j}(\tau^j)^*(\hat{\gamma}^*) = \frac{r}{k} \sum_{j=1}^k {\zeta}^{j}(\tau^j)^*(\hat{\gamma}''^*)$$
between two absolute cohomology classes in $H^1(\hat{X};\mathbb{Q}(\zeta))_{\bar{\zeta}}$. Taking the conjugate of both sides yields $$\sum_{j=0}^{k-1}\bar{\zeta}^{j}(\tau^*)^{j}(\hat{\gamma}''^*) = \bar{r}v. $$ Therefore, we have arrived at a contradiction since $\hat{\gamma}'$ pairs trivially with $v$.
\end{proof}

The disjointness from the previous Lemma will allow us to easily collapse all cylinders $\mathcal{N}$-parallel to (and distinct from) some cylinder in the pre-image of a Euclidean cylinder. Afterwards, there will be only a single cylinder remaining in its $\mathcal{M}$-parallel equivalence class on some surface in $\mathcal{N}$, and \cite[Theorem 7.1]{Wri15a} will imply arithmeticity.

\begin{lemma} \label{lemma_simplecylinder}
Suppose $k$ is prime. If there is some surface with a Euclidean cylinder in $K$, then $\mathcal{F(M)}$ is arithmetic.
\end{lemma}
\begin{proof}
Suppose $(X,\xi)$ is a surface in $K$ with a Euclidean cylinder $C$ and which has been perturbed so that only cylinders with hat homologous core curves are parallel (up to an angle of $\frac{2\pi}{k}$). By Lemma \ref{lemma_parallelsaddle}, such a perturbation always exists. Lift $(X,\xi)$ to $\mathcal{N}$ and consider some cylinder $\hat{C} \subset (\hat{X},\hat{\omega})$ in the pre-image of $C$. Up to rotation, we may assume that $\hat{C}$ is horizontal. Forgetting finitely many marked points does not increase (and when there are multiple cylinders, sometimes decreases) the number of parallel cylinders in some direction. If there are no other cylinders parallel to $\hat{C}$ on $(\hat{X},\hat{\omega})$, then there is a lone horizontal cylinder on $\mathcal{F}((\hat{X},\hat{\omega}))$. In this case, $\mathcal{F(M)}$ is arithmetic by \cite[Theorem 7.1]{Wri15a}. 

Next, assume there are cylinders $\hat{C}_{1},...,\hat{C}_{m}$ with respective core curves $\hat{\gamma}_1,...,\hat{\gamma}_m$ and heights $h_1,...,h_m$ that are parallel to (and distinct from) $\hat{C}$ on $(\hat{X},\hat{\omega})$. By Lemma \ref{lemma_sameintersectionsimple}, all of the core curves of $\hat{C}$ and the collection $\tau^j(\hat{C}_i)$ for every $i$ and $j$ are pairwise disjoint. Thus, any cylinder deformation about $\tau^j(\hat{C}_{i})$ independently will not change the circumference, height, or direction of any other cylinder in this collection. First, slightly shear the cylinders in the direction $$\sum_{j=0}^{k-1}h_i\zeta^jI(\tau^j(\hat{\gamma}_i)) =\sum_{j=0}^{k-1}h_i\bar{\zeta}^j(\tau^*)^j(I(\hat{\gamma}_i))$$ in $T\mathcal{N}$ for each $i$ so the closure of $\hat{C}_i$ does not contain a vertical saddle connection. Then, collapse each cylinder $\tau^j(\hat{C}_i)$ in the direction $$-i \sum_{i=1}^m\sum_{j=0}^{k-1}h_i\zeta^jI(\tau^j(\hat{\gamma}_i))=  -i\sum_{i=1}^m\sum_{j=0}^{k-1}h_i\bar{\zeta}^j(\tau^*)^j(I(\hat{\gamma}_i))$$ in $T\mathcal{N}$. This path can be imagined as the lift of the path in $K$ defined by shearing and collapsing each $\pi(\hat{C}_i)$ on $(X,\xi)$. See Figure \ref{fig: Euclidcylproof}. The resulting surface $(\hat{X}',\hat{\omega}')$ lies in the interior of $\mathcal{N}$ because the length of no saddle connection went to zero. Via marking the surfaces along this path, we can identify $\hat{C}$ and each $\hat{\gamma}_i$ on $(\hat{X}',\hat{\omega}')$. The image of $\hat{C}_i$ (which is the image of $\hat{\gamma}_i$) on $(\hat{Y},\hat{\eta})$ is a concatenation of horizontal saddle connections. 
\begin{figure} \
    \begin{center}
    {\includegraphics[width=0.90\textwidth]{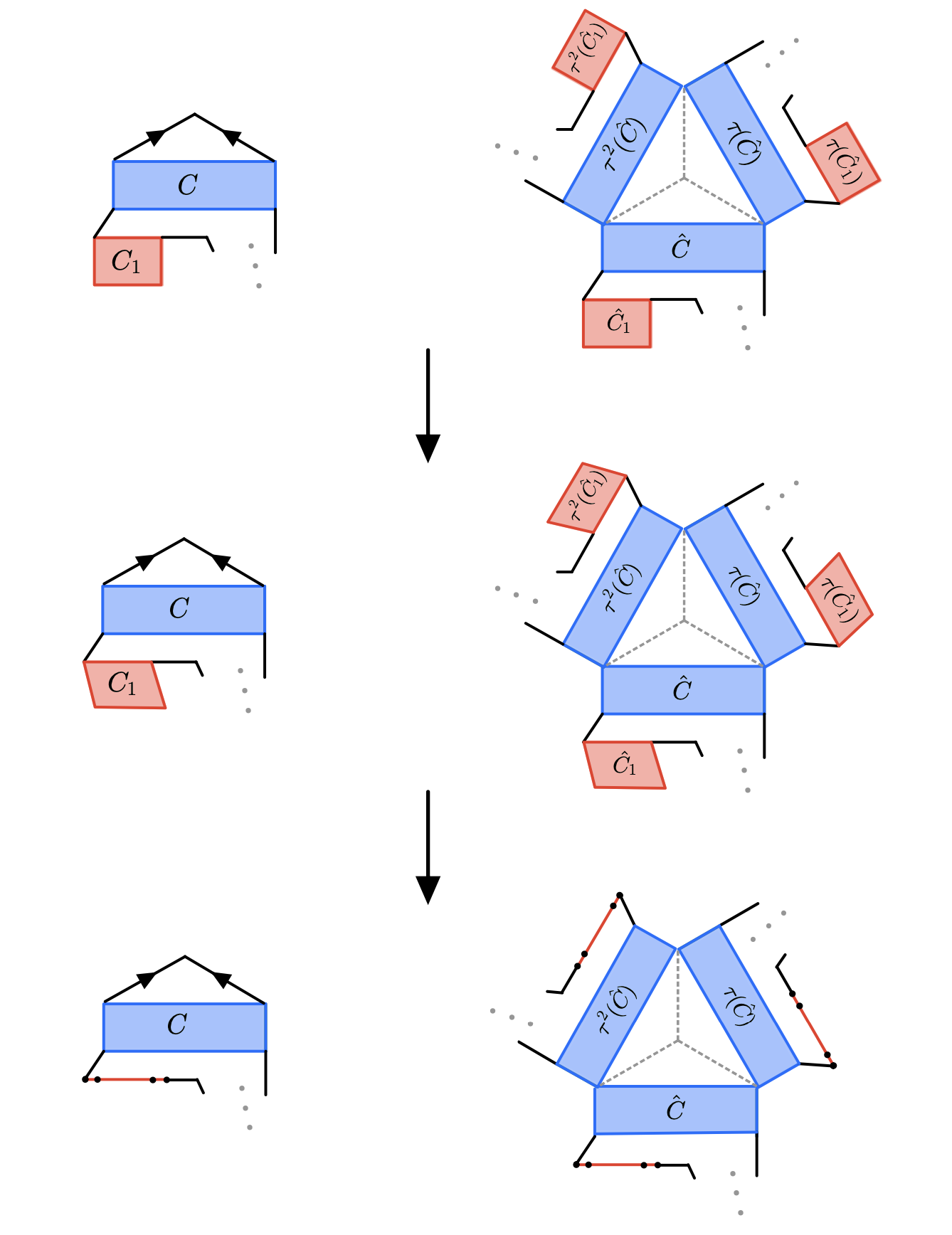}}
    \end{center}
    \caption{The shearing and collapsing of cylinders as in the proof of Lemma \ref{lemma_simplecylinder}.}
    \label{fig: Euclidcylproof}
\end{figure}

Suppose there was a horizontal cylinder $\hat{C}'$ distinct from $\hat{C}$ with a core curve $\hat{\gamma}'$ on $(\hat{X}',\hat{\omega}')$. If $\hat{C}'$ was disjoint from each $\tau^j(\hat{\gamma}_i)$, then we can trace it back to a horizontal cylinder on $(\hat{X},\hat{\omega})$ since our deformation to $(\hat{X}',\hat{\omega}')$ is an isometry away from each $\tau^j(\hat{C}_i)$. However, we had collapsed all cylinders on $(\hat{X},\hat{\omega})$ that were parallel to $\hat{C}$. This implies $\hat{C}'$ has at least one non-trivial intersection with some $\tau^j(\hat{\gamma}_i)$. Any small neighborhood $U$ of $(\hat{X}',\hat{\omega}')$ in $\mathcal{N}$ contains a surface where $\hat{\gamma}$, $\hat{\gamma}'$, and each $\hat{\gamma}_i$ are cylinder core curves. For instance, we can choose a surface along our deforming path to $(\hat{X}',\hat{\omega}')$. On this surface of course, $\hat{C}$ and each $\hat{C}_i$ are $\mathcal{N}$-parallel. Lemma \ref{lemma_sameintersectionsimple} implies $\hat{C}'$ cannot be in the same $\mathcal{N}$-parallel equivalence class because its core curve intersects a core curve of some $\tau^j(\hat{C}_i)$. We conclude there cannot be cylinders $\mathcal{N}$-parallel, and moreover $\mathcal{M}$-parallel, to $\hat{C}$ on $(\hat{X}',\hat{\omega}')$. Thus, there is a lone horizontal cylinder in its $\mathcal{F(M)}$-parallel equivalence class on $\mathcal{F}((\hat{X}',\hat{\omega}'))$. By \cite[Theorem 7.1]{Wri15a}, $\mathcal{F(M)}$ is arithmetic.
\end{proof}

The converse of Lemma \ref{lemma_simplecylinder} is not always true. Mirzakhani-Wright \cite{MW18} argued $\mathcal{M}$ is full rank when $(k,g) = (3,0)$ by using Theorem \ref{lemma_zetaeigenspace} and the fact $H^1(\hat{X};\mathbb{C})_1$ is empty when $g =0$. Theorem \ref{theorem_fullrank} implies $\mathcal{F(M)}$ is also arithmetic, since all components and hyperelliptic loci are arithmetic. Genus zero strata of $k$-differentials with three singularities have dimension one; we can rescale the differential, and that is it. In general, the projectivized stratum $\mathbb{P}\Omega^k\mathcal{M}_g(\mu)$ is defined as the quotient of $\Omega^k\mathcal{M}_g(\mu)$ by the $\mathbb{C}^*$-action which act by rescaling the differential. Therefore, $\mathbb{P}\Omega^3\mathcal{M}_0(m_1,m_2,m_3)$ is a single point after noting genus zero strata are also irreducible. If there was a Euclidean cylinder on a surface somewhere in the stratum, we could pinch its core curve and diverge off to infinity in $\mathbb{P}\Omega^3\mathcal{M}_0(m_1,m_2,m_3)$, so this space would not be compact. For that reason, $K=\Omega^3\mathcal{M}_0(m_1,m_2,m_3)$ is a counterexample.

Projectivized strata are used in this subsection, so now we discuss a compactification of $\mathbb{P}\Omega^k\mathcal{M}_g(\mu)$ and so forth. Let $\Omega^k\mathcal{M}_{g,n}(\mu)$ (or $\mathbb{P}\Omega^k\mathcal{M}_{g,n}(\mu)$) be the stratum in which we label the singularities (and then projectivize). Let $\text{Sym}(\mu)$ be the subgroup of the permutation group of the singularities which only permutes singularities of the same prescribed order. Then, we obtain $\Omega^k\mathcal{M}_g(\mu) = \Omega^k\mathcal{M}_{g,n}(\mu) /  \text{Sym}(\mu)$. What is dubbed \textit{the moduli space of multi-scale $k$-differentials} $\Xi^k \overline{\mathcal{M}}_{g,n}(\mu)$, which is constructed in \cite{CMM24}, is a generalization of the moduli space of multi-scale differentials $\Xi \overline{\mathcal{M}}_{g,n}(\mu)$ constructed in \cite{BCGGM24}. A point on the boundary $\Xi^k \overline{\mathcal{M}}_{g,n}(\mu) \backslash \Omega^k\mathcal{M}_{g,n}(\mu)$ is an enhanced level graph which encodes which curves are being pinched and the relative speeds of which subsurfaces are being crushed near the boundary, along with a twisted $k$-differential compatible with the level graph and an equivalence class of prong-matchings. Moreover, $\Xi^k  \overline{\mathcal{M}}_{g,n}(\mu) / \text{Sym}(\mu)$ contains $\Omega^k\mathcal{M}_g(\mu)$. The $\mathbb{C}^*$-action extends to $\Xi^k \overline{\mathcal{M}}_{g,n}(\mu)$ and the projectivized variety $\mathbb{P}\Xi^k \overline{\mathcal{M}}_{g,n}(\mu)$ is a compactification of $\mathbb{P}\Omega^k\mathcal{M}_{g,n}(\mu)$. 

We direct the reader to \cite{CMM24} and \cite{BCGGM24} for more precise details and definitions. Also, see \cite{Doz24} for an overview of different compactifications for strata of translation surfaces and how they compare. Throughout this paper, we will use the operations of \textit{bubbling a handle} and \textit{breaking up a zero} which are detailed nicely in \cite[Section 3]{CG22} for $(1/k)$-translation surfaces. Below, we show that $\Xi^k \overline{\mathcal{M}}_{g,n}(\mu) \backslash \Omega^k\mathcal{M}_{g,n}(\mu)$ is non-empty i.e. projectivized components of positive dimension are never compact. The author thanks Paul Apisa for conversations which led to the proof of Lemma \ref{lemma_non-compact}.

\begin{lemma} \label{lemma_non-compact}
   Suppose that $\mathbb{P}K$ is a component of $\mathbb{P}\Omega^k\mathcal{M}_{g,n}(\mu)$ and has positive dimension. Then $\mathbb{P}K$ is non-compact.
\end{lemma}
\begin{proof}
 We first note $\mathbb{P}K$ always has positive dimension unless it is a genus zero stratum with at most three singularities. In genus zero, $\mathbb{P}\Omega^k\mathcal{M}_{0,n}(\mu)$ is isomorphic to $\mathcal{M}_{0,n}$, so we can always merge a pair of singularities when there are at least four. The cotangent bundle is trivial on a genus one surface, so $\mathbb{P}K$ is isomorphic to a component of $\mathbb{P}\mathcal{H}_1(\mu)$ which are never compact. Projectivized components of strata of abelian differentials are non-compact, so the same follows for $k$-th powers of abelian differentials. We now move onto higher genus and assume $\mathbb{P}K$ consists not of $k$-th powers of abelian differentials.

First assume $g>2$ or $g = 2$ and there are branched points on the holonomy covers. Let $\mathcal{L}$ be the unit area locus inside the unprojectivized holonomy covers $\mathcal{N}$. If $\mathbb{P}K$ is compact, then so is $\mathcal{L}$ since it is an $S^1$-fiber bundle over $\mathbb{P}K$. Let $f$ be the function on $\mathcal{L}$ which computes the length of the shortest saddle connection on a surface. Compactness would imply there is a surface in $\mathcal{L}$ which realizes the minimum of $f$. Call that surface and saddle connection $(\hat{X},\hat{\omega})$ and $s$ respectively. To obtain a contradiction, we will show there is a deformation of $(\hat{X},\hat{\omega})$ in $\mathcal{N}$ where the period of $s$ is fixed but the area of the surface increases. The length of $s$ then gets proportionally smaller to the area. Therefore, re-normalizing the area will leave $s$ shorter than before in $\mathcal{L}$.

Let $H^{1,0}(\hat{X};\mathbb{C})_{\zeta}$ be the space of holomorphic one-forms in $H^1(\hat{X};\mathbb{C})_{\zeta}$. The dimension of $H^{1,0}(\hat{X};\mathbb{C})_{\zeta}$ is at least two when $g>2$ or $g = 2$ and there are branched points (see \cite[page 273-277]{FK92}). Hence, there is a $v \in H^1(\hat{X}, \Sigma(\hat{\omega});\mathbb{C})_{\zeta}$ which pairs trivially with $s$ and projects to a non-zero class in $H^{1,0}(\hat{X};\mathbb{C})_{\zeta}$. 

We can carry over any cohomology class $\omega$ on $\hat{X}$ to a cohomology class on a nearby surface $S$. We use the same notation for two cohomology classes if one is carried over from the other on a nearby surface, slightly abusing notation. Recall the Hodge bilinear form (on projected classes) $$\langle \omega_1, \omega_2 \rangle = \frac{i}{2}\int_S \omega_1 \wedge \bar{\omega}_2$$ which is positive-definite on $H^{1,0}(S)$ and where $S$ is a surface nearby $\hat{X}$. The Riemann bilinear relation expresses the norm of $\omega_1$ and $\omega_2$ in terms of the values obtained by pairing $\omega_1$ and $\omega_2$ with elements of a symplectic basis of homology. Hence, the norm does not depend on the surface $S$ nearby $\hat{X}$.

Recall $\langle \hat{\omega}, \hat{\omega} \rangle = 1$ because $(\hat{X},\hat{\omega})$ is of unit area. If we set $\hat{M}_t = (\hat{X},\hat{\omega}) + vt$ for $t$ small, we obtain that $$\text{Area}(\hat{M}_t) = \langle \hat{\omega} + vt,\hat{\omega}+ vt \rangle= 1 + \left(\langle \hat{\omega}, v \rangle + \langle v, \hat{\omega} \rangle \right)t + \langle v, v \rangle t^2.$$ This implies $\frac{d^2}{dt^2}\text{Area}(\hat{M}_t) > 0$, so $(\hat{X},\hat{\omega})$ is not a local maximum for the area function. Hence, there is a value for $t$ arbitrary close to $0$ for which $\text{Area}(\hat{M}_t) > 1$ as desired. Since the period of $s$ did not change, we are done with this case.

Now assume that $g = 2$ and there are no branched points i.e. the orders of all the singularities are multiples of $k$. Lemma \ref{lemma_projection} implies the kernel of the projection map to absolute cohomology then has positive dimension, so there are deformations in $\mathcal{N}$ which preserve the absolute periods. Moreover, there are arcs bounded by singularities whose periods are independent of the absolute periods in $\mathcal{N}$. Such arcs are called rel-shrinkable. Let $f$ now be the function on $\mathcal{N}$ which inputs a surface and outputs the length of the shortest rel-shrinkable arc. The geodesic representative of a rel-shrinkable arc is a concatenation of saddle connections with at least one being rel-shrinkable. Hence, the shortest rel-shrinkable arc must always be a saddle connection, and the length is locally given by the magnitude of its period. If $\mathcal{L}$ is compact, then there is a surface $(\hat{X},\hat{\omega})$ which attains the minimum of $f$ because $f$ is continuous. Let $s$ be shortest rel-shrinkable arc on $(\hat{X},\hat{\omega})$. Then we can slightly perturb $(\hat{X},\hat{\omega})$ so that the saddle connection $s$ shrinks and the absolute periods (and hence area) are preserved. We arrive at another surface in $\mathcal{L}$ where $s$ is shorter, and hence we have a contradiction.
\end{proof}

We will now prove when $g >2$, every component of $\Omega^k\mathcal{M}_{g}(\mu)$ has a surface with a Euclidean cylinder. This is performed using induction with base case $g=3$. To avoid subtleties, we choose to work in the compactification $\Xi^k \overline{\mathcal{M}}_{g,n}(\mu)$ which is what has been studied. However, the existence of a cylinder does not depend on whether or not we label singularities, so converting our result from $\Omega^k\mathcal{M}_{g,n}(\mu)$ to $\Omega^k\mathcal{M}_g(\mu)$ is not an issue. Suppose $\mathcal{N}$ consists of genus $\hat{g}$ surfaces with $\hat{n}$ singularities. 

\begin{lemma} \label{lemma_degeneratedgraph}
Suppose that $\mathbb{P}K$ is a component of $\mathbb{P}\Omega^k\mathcal{M}_{g,n}(\mu)$ and has positive dimension. Then there is a point on the boundary of $\mathbb{P}K$ whose level graph $\Gamma$ either has a horizontal edge or is such that 
\begin{enumerate}
\item there are no horizontal edges,
\item every vertex represents a genus zero $(1/k)$-translation surface with at most three singularities, 
\item and the locally maximal vertices represent primitive $k$-differentials.
\end{enumerate}
\end{lemma}
\begin{proof}
By Lemma \ref{lemma_non-compact}, any projectivized component of a stratum without higher order poles can be degenerated unless it is a point. Moreover, \cite[Theorem 1.1]{Chen19} implies the same for strata with higher order poles. Therefore, $\overline{\mathbb{P}K}$ intersects the boundary $\mathbb{P} \Xi^k \overline{\mathcal{M}}_{g,n}(\mu) \backslash \mathbb{P}\Omega^k\mathcal{M}_{g,n}(\mu)$. Additionally, we can keep degenerating within $\overline{\mathbb{P}K}$ until it is maximally degenerated. That is, the multi-scale $k$-differential lives in a dimension zero subspace of a product of strata cut out by the Global $k$-Residue Condition (G$k$RC) (see \cite[Definition 1.4]{BCGGM19}).

Assume that the level graph does not have a horizontal edge. The $(1/k)$-translation surfaces represented by locally maximal vertices on $\Gamma$ have no constraints imposed by the G$k$RC, so they necessarily live in strata of dimension zero. By a simple dimension count, the only such strata are of $k$-differentials over $\hat{\mathbb{C}}$ with at most three singularities. There cannot be higher order poles because none were prescribed and there are no upward vertical edges, and in our case horizontal edges, on locally maximal vertices. Because there are no holomorphic abelian differentials on $\hat{\mathbb{C}}$, the $k$-differentials associated to locally maximal vertices cannot be $k$-th powers of abelian differentials. This implies the G$k$RC is trivial everywhere (by satisfying \cite[Definition 1.4 (ii)]{BCGGM19}), and the ambient stratum of each irreducible component is of genus zero with at most three singularities.
\end{proof}

Though smoothing out is not always a local operation, if there is a Euclidean cylinder on a multi-scale $k$-differential, we may always stretch it out sufficiently before smoothing so that this cylinder \textit{persists} in the interior, i.e. if we mark this cylinder along these deformations of stretching and then smoothing, it remains a cylinder at the end.

\begin{lemma} \label{lemma_cylpersists}
    Let $C$ be a Euclidean cylinder on a multi-scale $k$-differential on the boundary of a component $K$ of $\Omega^k\mathcal{M}_g(\mu)$. There is a surface in $K$ in which $C$ persists.
\end{lemma}
\begin{proof}
    Suppose that $C$ lives on the irreducible component $(X,\xi)$ on the multi-scale $k$-differential. Stretch out $C$ to infinity so that the limit becomes two poles of order $-k$ and matching residues which get identified. This corresponds to a horizontal edge $h$ on the new level graph. We can then smooth out all the nodes other than the one corresponding to $h$ to obtain a new multi-scale $k$-differential (see \cite[Section 3.1]{CMM24}). Smoothing out a horizontal edge alone always creates a Euclidean cylinder, so here smoothing out the remaining node corresponding to $h$ realizes $C$ inside $K$. 
\end{proof}

The symbol $\rightsquigarrow$ is used to denote a degeneration of a level graph. Given a level graph $\Gamma$, the subgraph $\Gamma_{<L}$ (or resp.  $\Gamma_{>L}$, $\Gamma_{\leq L}$, $\Gamma_{\geq L}$) of $\Gamma$ consists of the vertices below level $L$ (resp. above level $L$, at level $L$ and below, at level $L$ and above) and all edges that connect them. When we \textit{undegenerate} $\Gamma_{<L}$ (or resp.  $\Gamma_{>L}$, $\Gamma_{\leq L}$, $\Gamma_{\geq L}$), we collapse all the edges of $\Gamma_{<L}$ (resp.  $\Gamma_{>L}$, $\Gamma_{\leq L}$, $\Gamma_{\geq L}$) and afterwards pull all the remaining vertices of this subgraph to level $L-1$ (resp. $L+1$, $L$, $L$).

\begin{lemma} \label{lemma_genus3}
If $K \subset \Omega^k\mathcal{M}_3(\mu)$, then there is a surface in $K$ with a Euclidean cylinder. Moreover, $\mathcal{F(M)}$ is arithmetic when $k$ is prime by Lemma \ref{lemma_simplecylinder}.
\end{lemma}
\begin{proof}
The image of $K$ in $\mathbb{P}\Omega^k\mathcal{M}_{3,n}(\mu)$, which we will call $\mathbb{P}K$, is always of positive dimension since $g = 3$. Therefore, let $\Gamma$ be the level graph from Lemma \ref{lemma_degeneratedgraph} on the boundary of $\mathbb{P}K$. Any horizontal edge in $\Gamma$ represents a pinched Euclidean cylinder. If $\Gamma$ has a horizontal edge, we know there is a Euclidean cylinder somewhere in the interior of the stratum and are done. 

Next, assume that $\Gamma$ is the graph from Lemma \ref{lemma_degeneratedgraph} without horizontal edges. In general, the genus of the underlying stable Riemann surface is equal to the sum of the genus on each irreducible component plus the first Betti number $b_1$ of its level graph. Because no irreducible component has positive genus in our case, $b_1(\Gamma) = 3$. 

To prove the Lemma from here, we will show by undegenerating and degenerating $\Gamma$ that there is a point on the boundary of $\mathbb{P}K$ which has a genus one irreducible component with at least one higher order pole. We will then use \cite[Theorem 3.12]{CG22} to show we can perturb the genus one surface to get a Euclidean cylinder and smooth out the multi-scale $k$-differential into the interior of $\mathbb{P}K$.

Call a vertex of $\Gamma$ a \textit{peak vertex} if it is the highest vertex of some simple cycle in $\Gamma$. Let $L$ be the lowest level that contains a peak vertex. Consider the subgraph $\Gamma_{<L}$ of $\Gamma$ defined above. Because there are no simple cycles below level $L$, $\Gamma_{<L}$ is a disjoint union of trees. We first undegenerate $\Gamma_{<L}$ to create a graph $\Gamma'$ whose lowest level is level $L-1$. We do this undegeneration only to simplify the rest of the argument. The G$k$RC remains trivial since the locally maximal vertices of $\Gamma'$ are a subset of those that were on $\Gamma$ and thus do not represent $k$-th powers of abelian differentials. At level $L-1$, $\Gamma'$ is a collection of vertices representing genus zero irreducible components. Choose some peak vertex $V$ at level $L$. 

For the first case, suppose $V$ has three downward edges connected to the same vertex at level $L-1$. The restriction of valance at most three on vertices in $\Gamma'_{\geq L}$ implies $V$ is locally maximal. Thus, we introduce level 1 and pull $V$ here while preserving the dual graph, partial order on the vertices, and enhancements. We will name this new graph $\Gamma''$. Undegenerating $\Gamma''_{<1}$, we obtain a two-level graph $\Gamma'''$ with two vertices connected by three edges. Because the genus of the irreducible component represented by $V$ is zero and $b_1(\Gamma''')=2$, the lower vertex represents a genus one surface with three poles of order at least $k+1$. See Figure \ref{fig: case1}.
\begin{figure}
    \begin{center}
    {\includegraphics[width=1.05\textwidth]{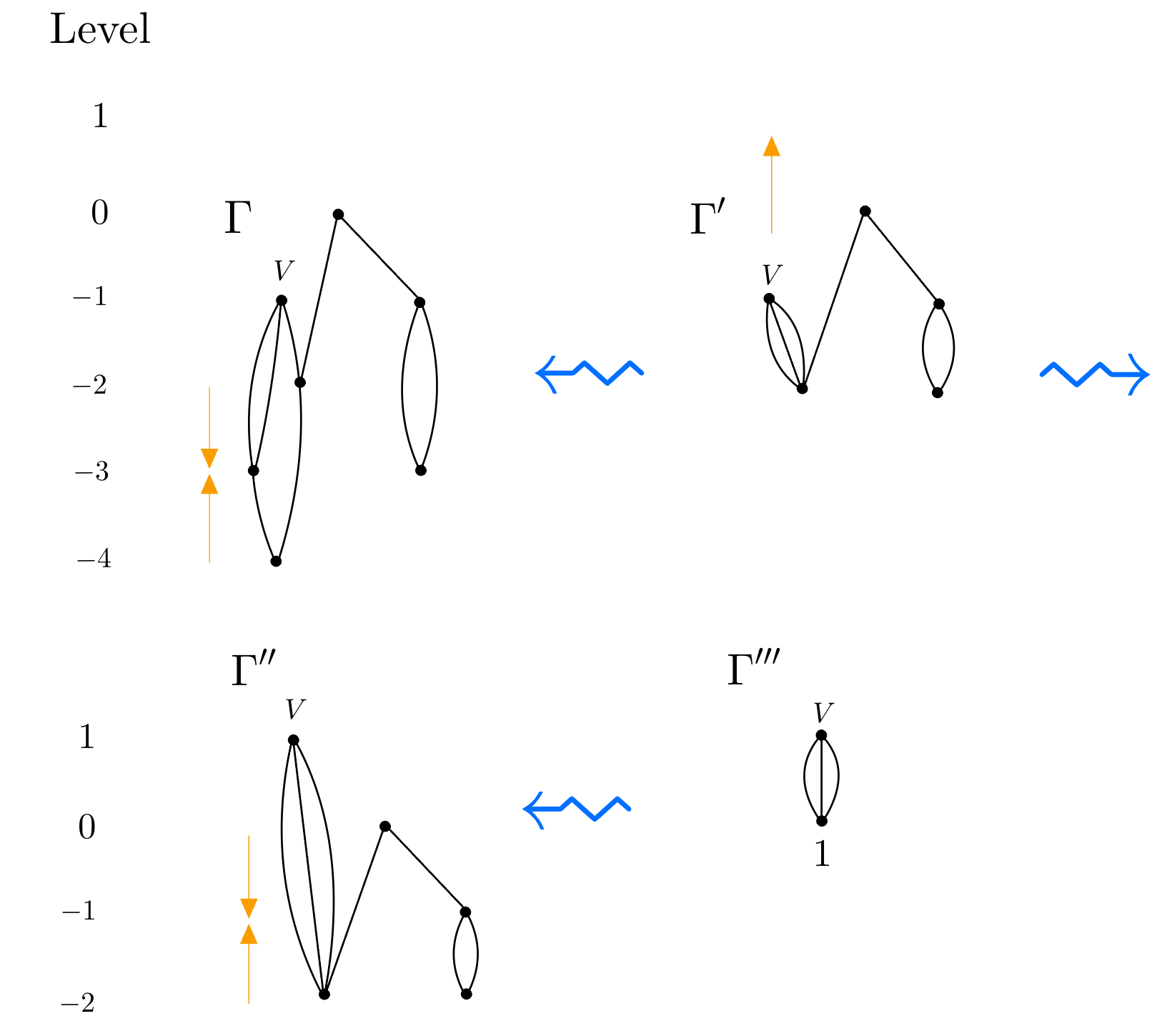}}
    \end{center}
    \caption{Case one in the proof of Lemma \ref{lemma_genus3}.}
    \label{fig: case1}
\end{figure}

The other case is $V$ has two downwards edges to the same vertex at level $L-1$. We then introduce level $L-0.5$ and pull $V$ here while preserving the dual graph, partial order on the vertices, and enhancements. Call this new graph $\Gamma''$ and note $b_1(\Gamma''_{\leq L-0.5}) = 1$ because $V$ is the only peak vertex of $\Gamma''_{\leq L-0.5}$. Undegenerating $\Gamma''_{\leq L-0.5}$ creates a vertex representing a genus one surface with at least one pole of order at least $k+1$. These poles are represented by vertical edges connecting this vertex to the rest of the graph (and exist because the graph is connected). See Figure \ref{fig: case2}.
\begin{figure}
    \begin{center}
    {\includegraphics[width=1.05\textwidth]{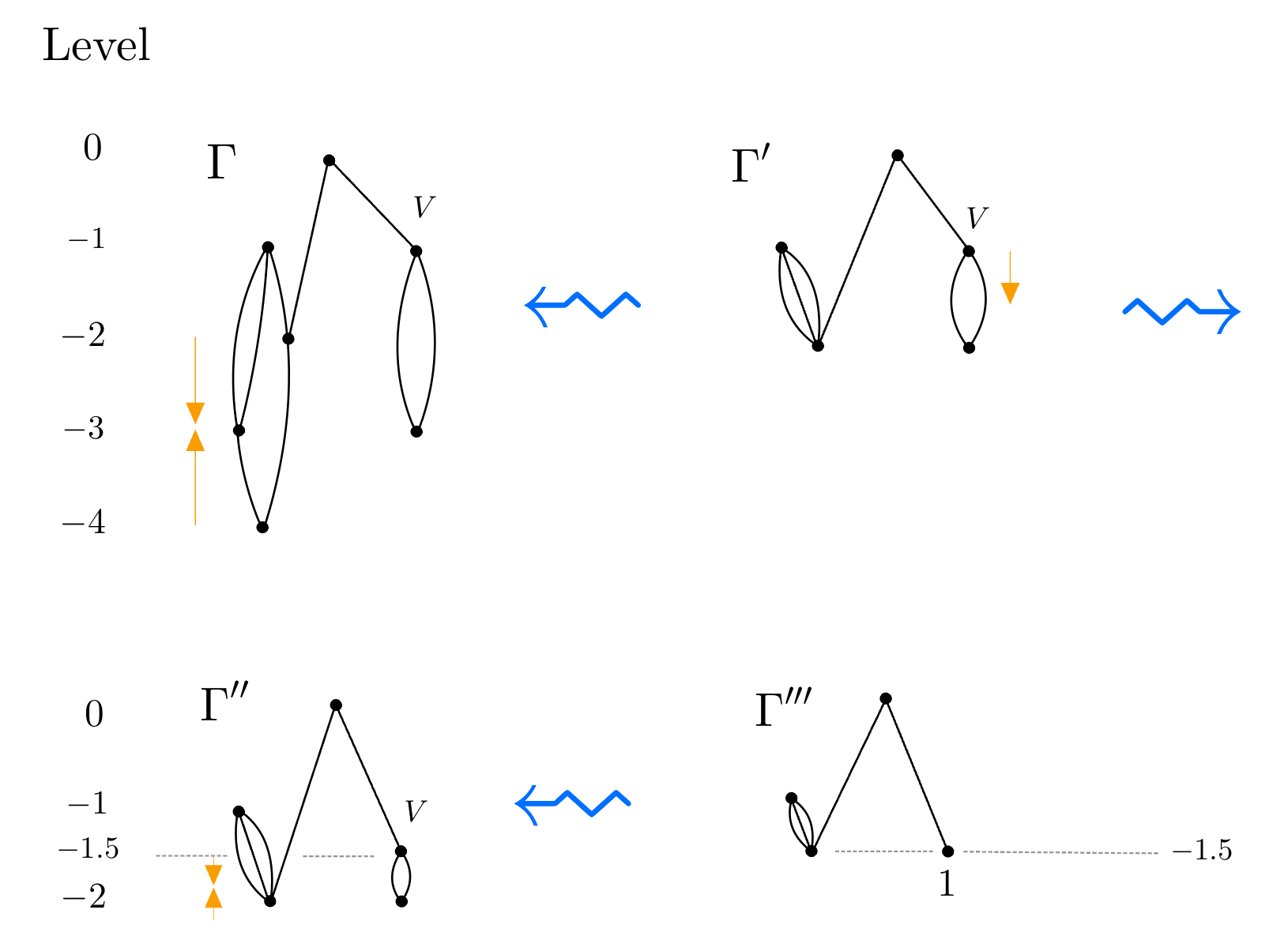}}
    \end{center}
    \caption{Case two in the proof of Lemma \ref{lemma_genus3}.}
    \label{fig: case2}
\end{figure}

In either case, call the genus one surface $(X,\xi)$. If a genus one surface has a pole of order at least $k+1$, then the sum of the poles of orders less that $k$ and zeros must be at least $k+1$. By \cite[Theorem 3.12]{CG22}, the ambient component of the stratum of $(X,\xi)$ contains a surface $(X',\xi')$ obtained by bubbling a handle from a surface in some genus zero stratum. In particular, $(X',\xi')$ contains a Euclidean cylinder which emerged from smoothing out a horizontal node from the bubbling a handle operation. In both cases, the locally maximal vertices on the new graph are a subset of those that were on $\Gamma$, and the projectivized $k$-differentials associated to them remain unchanged throughout the degenerations and undegenerations. Because none were the $k$-th power of an abelian differential, the G$k$RC remains trivial. Therefore, we can replace $(X,\xi)$ with $(X',\xi')$ while still being able to smooth the multi-scale $k$-differential into the interior of $\mathbb{P}\Omega^k\mathcal{M}_{3,n}(\mu)$. We can smooth out the multi-scale $k$-differential to obtain a welded surface in $\mathbb{P}K$ (and thus $K$) in which $C$ persists using Lemma \ref{lemma_cylpersists}. By Lemma \ref{lemma_simplecylinder}, $\mathcal{F(M)}$ is arithmetic.
\end{proof}

\begin{remark}
The proof of Lemma \ref{lemma_genus3} does not generalize to $g=2$ because after fully degenerating, we may obtain a unique lowest peak vertex as in case one. Pulling it up to level 1 and undegenerating the lower levels produces a genus zero, not one, vertex at the bottom. Therefore, we cannot apply \cite[Theorem 3.12]{CG22}. We have no control of how we fully degenerated in the first place, so we cannot avoid such a graph in our argument.
\end{remark}

\begin{lemma} \label{lemma_Marithmetic}
If $K \subset \Omega^k\mathcal{M}_g(\mu)$ where $g>2$, then there is a surface in $K$ with a Euclidean cylinder. Moreover, $\mathcal{F(M)}$ is arithmetic when $k$ is prime by Lemma \ref{lemma_simplecylinder}.
\end{lemma}

\begin{proof}
We will prove using induction that every component of $\mathbb{P}\Omega^k\mathcal{M}_{g,n}(\mu)$ has a surface with a Euclidean cylinder when $g >2$. From the previous Lemma, we know this holds for $g=3$. Thus, suppose any component of any stratum without higher order poles in $\mathbb{P}\Omega^k\mathcal{M}_{g-1}$ has a surface with a Euclidean cylinder. Consider the component $\mathbb{P}K$ of $\mathbb{P}\Omega^k\mathcal{M}_{g,n}(\mu)$. We first show without having horizontal edges by chance that there is a point on the boundary $\mathbb{P}\Xi^k \overline{\mathcal{M}}_{g,n}(\mu) \backslash \mathbb{P}\Omega^k\mathcal{M}_{g,n}(\mu)$ whose level graph has a vertex at the top representing a genus $g-1$ surface. Then by assumption, we can perturb the $g-1$ surface to have a Euclidean cylinder and smooth out the multi-scale $k$-differential.

We begin with the graph $\Gamma$ from Lemma \ref{lemma_degeneratedgraph}. If there are horizontal edges on the level graph, it indicates a pinched Euclidean cylinder, so there is one in the interior and we are done. Otherwise, we again obtain a graph $\Gamma$ with vertices which represent genus zero irreducible components with at most three singularities. Hence, the valence on each vertex is at most three. Moreover, none of the locally maximal vertices represent $k$-th powers of abelian differentials, so the G$k$RC is trivial by \cite[Definition 1.4 (ii)]{BCGGM19}.

Call a vertex a \textit{post vertex} if it is the lowest vertex of a simple cycle. Let $L$ be the lowest level of $\Gamma$ with a post vertex, and therefore $\Gamma_{\leq L}$ is a disjoint union of trees. Undegenerating $\Gamma_{>L}$ and $\Gamma_{\leq L}$ creates a two-level graph such that the bottom level, now called level $-1$, vertices represent genus zero surfaces. The G$k$RC remains trivial because every locally maximal vertex on this new graph was locally maximal or was merged with a locally maximal vertex on $\Gamma$. Hence, they cannot represent $k$-th powers of abelian differentials. At least one bottom level vertex has at least one pair of upwards edges $a$ and $b$ connected to the same vertex at the top level, now level 0. This vertex was a post vertex or merged with a post vertex in $\Gamma$. Call this vertex $V$ and consider the deformation of the graph which introduces level $-0.5$, pulls all bottom level vertices other than $V$ (if they exist) here, and undegenerates the top two levels. We are left with a two level graph $\Gamma'$ with $V$ as the only vertex at the bottom level. See Figure \ref{fig: 2levelgraph}. Moreover, $V$ still has the pair $a$ and $b$ connecting it to a vertex at the top level. If there are no other upward edges on $V$, then $\Gamma'$ must have only two vertices because it is connected and there are no horizontal edges connecting top level vertices. Furthermore, $b_1(\Gamma') =1$, and since $V$ represents a genus zero surface, the top vertex must represent a genus $g-1$ surface as desired. 

Otherwise, assume there is some edge $c$ distinct from $a$ and $b$. Once again, the vertices at the top level do not represent $k$-th powers of abelian differentials since neither did locally maximal vertices on $\Gamma$. Therefore, the G$k$RC remains trivial on $\Gamma'$ by \cite[Definition 1.4 (ii)]{BCGGM19}, and we have no conditions imposed on how we can degenerate $V$. Because the ambient stratum of the surface represented by $V$ is of genus zero, it is also isomorphic to $\mathcal{M}_{0,s}$. Given there are at least three upwards edges on $V$, there must also be at least one zero for the orders of singularities on the associated $k$-differential to sum to $-2k$. Thus $s \geq 4$, and there is a degeneration which brings together an arbitrary pair of singularities. In particular, there is a degeneration of $\Gamma'$ which collides the higher order poles corresponding to $a$ and $c$. On the new graph $\Gamma''$, there is a new vertex which has a new edge $d$ as its one downward edge to $V$ and $a$ and $c$ as its two upwards edges. Moreover, $V$ has one less upward edge and is still a post vertex because $a$, $b$, and $d$ form a simple cycle. We undegenerate all but the bottom level of $\Gamma''$ to create a new two-level graph in which $V$ is still a post vertex ($a$ gets collapsed, and $b$ and $d$ form a simple cycle). See Figure \ref{fig: procedure}. 

While observing the G$k$RC will still remains trivial throughout, we repeat this procedure of 
\begin{enumerate}[topsep=0pt]
    \item choosing a pair of edges which connect $V$ to the same top level vertex, 
    \item choosing an edge of this pair and an edge outside of this pair to collide,
    \item and undegenerating all but the bottom level
\end{enumerate} until $V$ has only two upwards edges on a two-level graph $\bar{\Gamma}$. Figure \ref{fig: procedure} demonstrates performing this procedure twice. Notice our choices avoid us colliding two edges which may be a lone pair connecting $V$ to another vertex. Therefore, $V$ remains a post vertex throughout. The connectivity of $\bar{\Gamma}$ implies there is a single vertex at the top. Since $b_1(\bar{\Gamma}) = 1$ and $V$ still represents a genus zero surface, the top level vertex represents a genus $g-1$ surface as desired.
\begin{figure}
    \begin{center}
    {\includegraphics[width=1\textwidth]{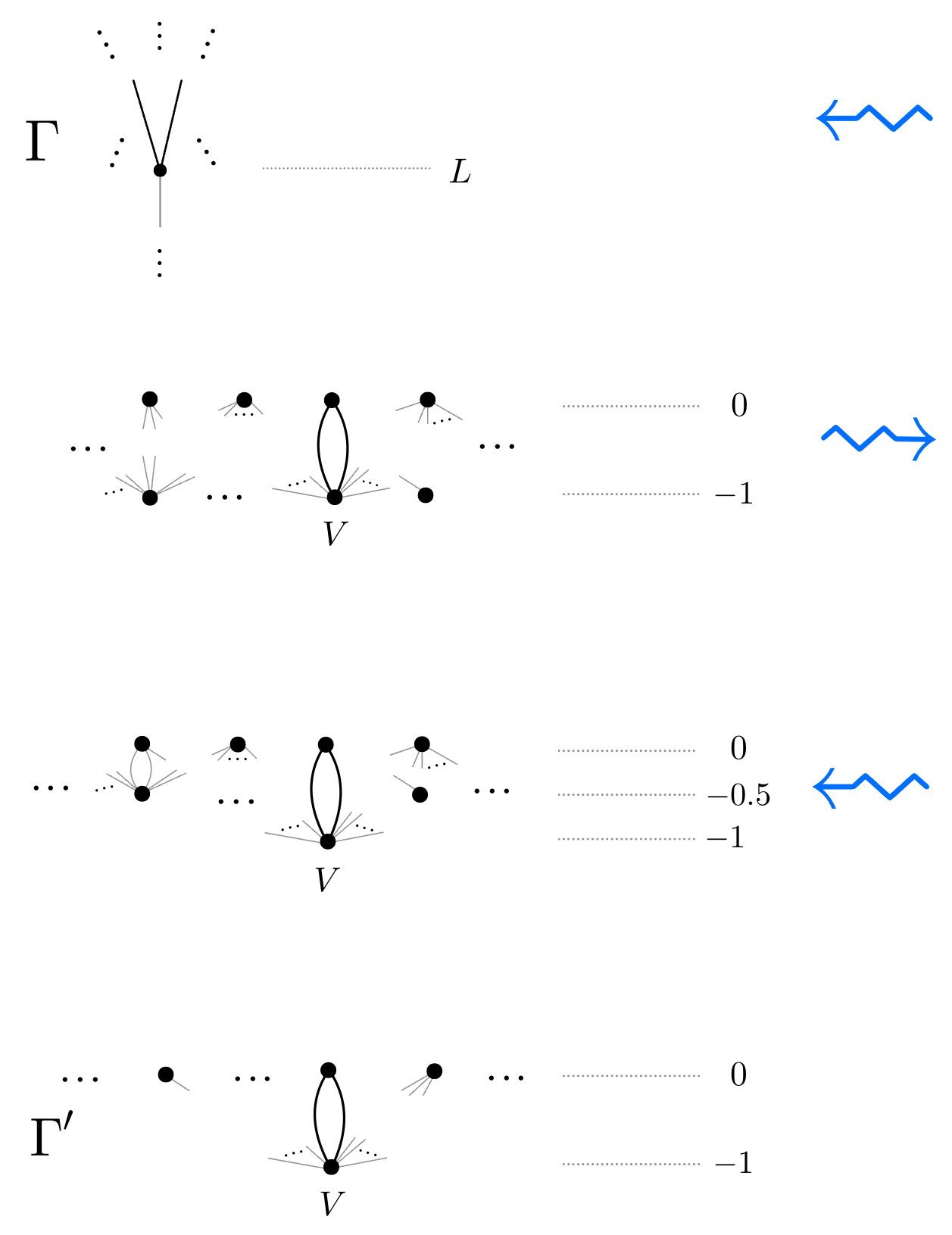}}
    \end{center}
    \caption{A series of undegenerating and degenerating the level graph in Lemma \ref{lemma_Marithmetic} to create a two-level graph with a single vertex on the bottom level.}
    \label{fig: 2levelgraph}
\end{figure}

\begin{figure}
    \begin{center}
    {\includegraphics[width=.9\textwidth]{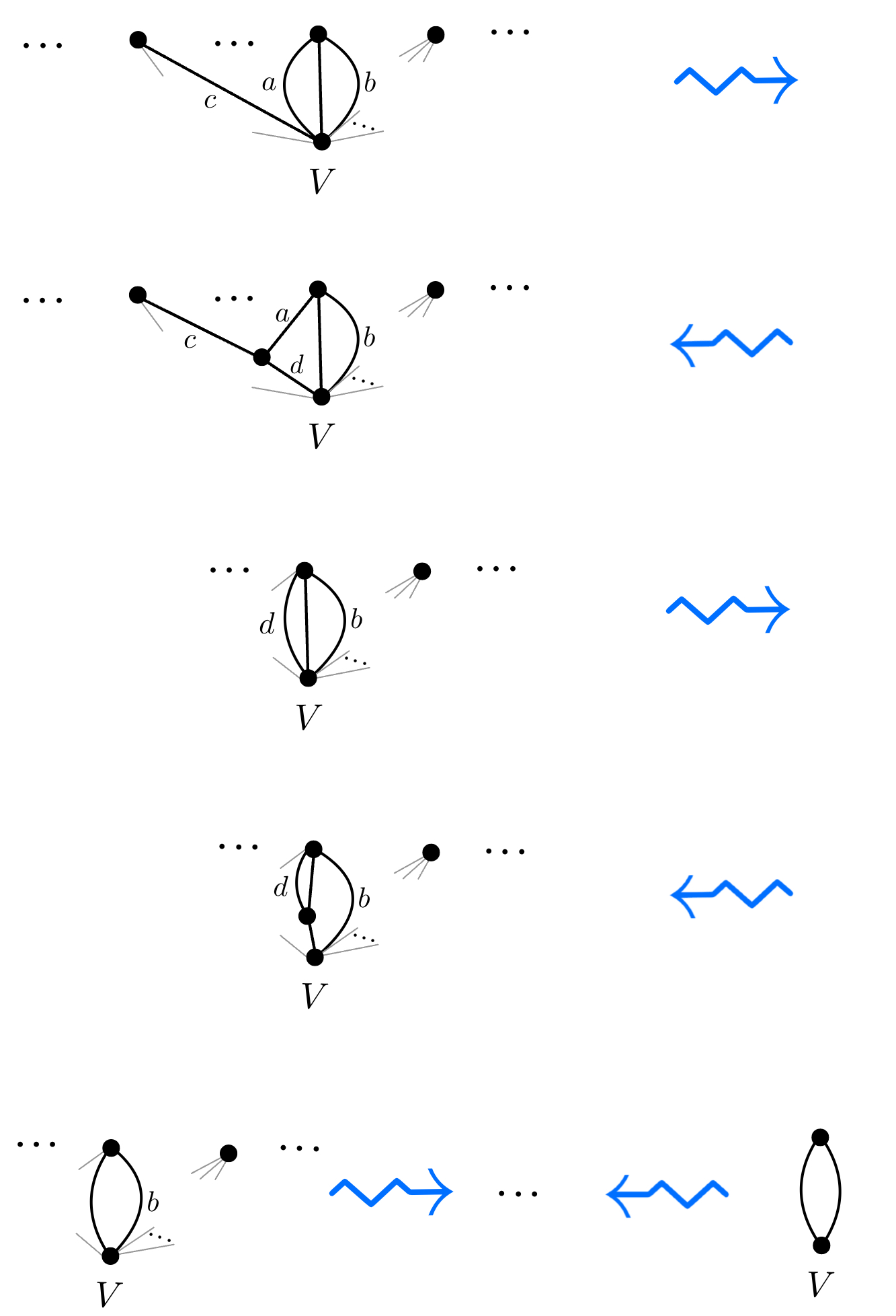}}
    \end{center}
    \caption{The procedure in Lemma \ref{lemma_Marithmetic} performed twice to decrease the number of vertical edges of post vertex $V$.}
    \label{fig: procedure}
\end{figure}
There is no G$k$RC anywhere (and at the top level in general), so we can continuously deform the genus $g-1$ surface until it has a Euclidean cylinder $C$. We can smooth out the multi-scale $k$-differential to obtain a welded surface in $\mathbb{P}K$ (and thus $K$) in which $C$ persists using Lemma \ref{lemma_cylpersists}. By Lemma \ref{lemma_simplecylinder}, $\mathcal{F(M)}$ is arithmetic. 
\end{proof}

In \cite{Api21}, we see $\mathcal{F(M)}$ could be either arithmetic or non-arithmetic when $K$ consists of genus zero surfaces. It remains an open question whether non-arithmeticity is possible when $K$ consists of genus one or two surfaces. Arithmeticity does however occur; many components in genus one and two can be realized from bubbling a handle followed by breaking up zeros. These components have a Euclidean cylinder, and thus its lift will have an arithmetic orbit closure by Lemma \ref{lemma_simplecylinder}. Lemma \ref{lemma_genus1cyl} provides conditions on the partition $\mu$ in which we know this is true in genus one. Arithmeticity could perhaps be proved with stronger adjacency results on strata of $k$-differentials than what is currently available.

The following Lemma is of course stronger than Lemma \ref{lemma_Marithmetic} but not needed to show arithmeticity. We will need it for the case $(k,g) = (3,3)$ in the proof of Lemma \ref{lemma_Mhighrank}. We say a cylinder is null-homologous if its core curves are null-homologous in absolute homology.

\begin{lemma} \label{lemma_notnullcyl}
If $K \subset \Omega^k\mathcal{M}_g(\mu)$ where $g>2$, then there is a surface in $K$ with a Euclidean cylinder which is not null-homologous. 
\end{lemma}
\begin{proof}
By Lemma \ref{lemma_Marithmetic}, there is a surface $(X,\xi)$ in $K$ with a Euclidean cylinder $C$. If $C$ does not separate the surface, we are done. By assuming otherwise, any saddle connection which intersects a core curve of $C$ once must not be closed i.e. connects two different singularities. First we perturb the surface so that on each boundary component of $C$, every saddle connection is hat homologous to each other. Hence, collapsing $C$ will not change the genus. Consider a saddle connection $s$ contained in the closure $\overline{C}$ that intersects its core curves once. Such saddle connections are called cross curves. By collapsing $C$ in the direction of $s$, we collide these singularities and obtain a $(1/k)$-translation surface $(X',\xi')$ of genus $g$ with less singularities. Thus, along this path we converge to a point on the boundary whose level graph has a top level vertex representing $(X',\xi')$. In the ambient component of $(X',\xi')$ in its stratum, Lemma \ref{lemma_Marithmetic} implies there is a surface with a Euclidean cylinder $C'$. If $C'$ separates the surface, we again perturb the surface so that each boundary component consists of only hat homologous saddle connections and collapse $C'$ in the direction of one of its cross curves. Meanwhile then, we are converging to another point on the boundary of $K$ whose top level vertex is a genus $g$ surface with even fewer singularities than originally.

We repeat this process of collapsing a separating Euclidean cylinder and re-applying Lemma \ref{lemma_Marithmetic} until a cylinder from Lemma \ref{lemma_Marithmetic} is not null-homologous or the stratum of the top level vertex is $\Omega^k\mathcal{M}_g(k(2g-2))$. If the latter happens, the cross curve of the Euclidean cylinder given to us by Lemma \ref{lemma_Marithmetic} is closed. Therefore, the cylinder is not null-homologous since its core curves intersect a closed curve once. 

Once we obtain a Euclidean cylinder which is not null-homologous, we smooth out the multi-scale $k$-differential into $K$ while preserving this cylinder using Lemma \ref{lemma_cylpersists}. Because we never lost genus while degenerating, the cylinder remains non-separating.
\end{proof}

\subsection{Classification of $\mathcal{M}$} \label{subsection_classification}
Using now that $\mathcal{F(M)}$ is arithmetic, we can show all primitive eigenspaces $H^1(\hat{X};\mathbb{C})_{\zeta^\ell}$ are contained in $p(T_{(\hat{X},\hat{\omega})}\mathcal{M})$ (Lemma \ref{lemma_eigenspacefullrank} below). By summing up the dimension of these eigenspaces, we will deduce $\mathcal{M}$ is high rank (Lemma \ref{lemma_Mhighrank}). 

\begin{lemma} \label{lemma_eigenspacefullrank}
Let $k>2$ be prime and $\ell \in \{1,...,k-1\}$. If $\mathcal{F(M)}$ is arithmetic, then $H^1(\hat{X};\mathbb{C})_{\zeta^\ell}$ is contained in $p(T_{(\hat{X},\hat{\omega})}\mathcal{M})$ for any $(\hat{X},\hat{\omega}) \in \mathcal{N}$.
\end{lemma}
\begin{proof}
Let $H^1$ be the flat subbundle over $\mathcal{F(M)}$ whose fiber over $(S,\omega) \in \mathcal{F(M)}$ is $H^1(S;\mathbb{C})$. Wright showed in \cite{Wri14} that there is a flat bundle $W$ defined over $\mathbb{Q}$ and, for each field embedding $\rho : \textbf{k}(\mathcal{F(M)}) \to \mathbb{C}$, a flat bundle $V_{\rho}$ which is Galois conjugate to $V_{\text{Id}} = p(T\mathcal{F(M)})$ so that $$H^1 = \left( \bigoplus_{\rho} V_{\rho} \right) \oplus W.$$ Because $\textbf{k}(\mathcal{F(M)}) = \mathbb{Q}$, this decomposition simplifies to $$H^1 = p(T\mathcal{F(M)}) \oplus W.$$ Moreover, $p(T\mathcal{F(M)})$ and $W$ defined over $\mathbb{Q}$ implies they are stable under all field automorphisms of $\mathbb{C}$. Consider any $\mathcal{F}((\hat{X},\hat{\omega})) \in \mathcal{F(N)}$ and $v \in H^1(\hat{X};\mathbb{Q}(\zeta))_{\zeta^\ell}$. Let $\varphi$ be the $\mathbb{Q}$-linear map that is the field automorphism in the Galois group of $\mathbb{Q}(\zeta)$ sending $\zeta^{\ell}$ to $\zeta$ (applied to each entry of a cohomology class). Seeing that $\tau^*$ is an integral operator, we obtain $$\tau^*(\varphi(v)) = \varphi(\tau^*(v)) = \varphi(\zeta^\ell v) =\zeta \varphi(v).$$ By Lemma \ref{lemma_zetaeigenspace}, $\varphi(v)$ is in the fiber of $p(T\mathcal{F(M)})$ over $\mathcal{F}((\hat{X},\hat{\omega}))$. Thus, $v$ is also contained in $p(T\mathcal{F(M)})$ as well. There is a basis for $H^1(\hat{X};\mathbb{C})_{\zeta^\ell}$ in $ H^1(\hat{X};\mathbb{Q}(\zeta))_{\zeta^\ell}$, so $H^1(\hat{X};\mathbb{C})_{\zeta^\ell}$ is contained in $p(T\mathcal{F(M)})$, and hence, $p(T\mathcal{M}
)$.
\end{proof}

For the component(s) of $\Omega^3\mathcal{M}_{2}(6)^{\text{prim}}$ or $\Omega^3\mathcal{M}_{2}(3,3)^{\text{prim}}$, a simple dimension count shows $\mathcal{M}$ is not automatically high rank if $p(T_{(\hat{X},\hat{\omega})}\mathcal{M})$ contains all the primitive eigenspaces. Therefore, we make the assumption below that $(k,g) \neq (3,2)$.

\begin{lemma} \label{lemma_Mhighrank}
Assume that $k>2$ is prime and $(k,g) \neq (3,2)$. If $\mathcal{F(M)}$ is arithmetic, then $\mathcal{M}$ is high rank.
\end{lemma}
\begin{proof}
If $\mathcal{F(M)}$ is arithmetic, then $H^1(\hat{X};\mathbb{C})_1$ is the only eigenspace not yet proved to be in $p(T_{(\hat{X},\hat{\omega})}\mathcal{M})$ after Lemma \ref{lemma_eigenspacefullrank}. We have the isomorphism $$H^1(\hat{X};\mathbb{C})_1 \cong  H_1(X;\mathbb{C})$$ so $H^1(\hat{X};\mathbb{C})_1$ is of dimension $2g$. Therefore, the largest possible deficit of $\text{rank}(\mathcal{M})$ from full rank is $g$. When $g=0$, we automatically arrive at full rank after Lemma \ref{lemma_eigenspacefullrank}. Using that $g \leq 1 + \frac{\hat{g}-1}{k} $ from the Riemann-Hurwitz formula, we compute $$\text{rank}(\mathcal{M}) \geq \hat{g} - g \geq \hat{g} -  \left( 1 + \frac{\hat{g} -1}{k}  \right)= \left(1- \frac{1}{k}\right)\hat{g} - \left(1 - \frac{1}{k}\right).$$ Therefore, $\mathcal{M}$ is high rank if $$\left(1- \frac{1}{k}\right)\hat{g} - \left(1 - \frac{1}{k}\right) \geq \frac{\hat{g}}{2}+1,$$ or equivalently, $$\left(\frac{1}{2}- \frac{1}{k}\right) \hat{g} \geq 2 - \frac{1}{k}.$$ Using again that $\hat{g} \geq 1 + k(g-1)$, it suffices for the inequality $$\left(\frac{1}{2}- \frac{1}{k}\right) \left(1 + k(g-1)\right) \geq 2 - \frac{1}{k}$$ to be satisfied. One can deduce after taking partial derivatives that as $g$ increases, only the left-hand side increases because $k>2$, and when $k$ increases, the left-hand side increases faster than the right-hand side. The inequality is satisfied when $(k,g) = (5,2)$ and $(k,g) = (3,4)$, so $\mathcal{M}$ is high rank when either $k >3$ is prime and $g >1$ or $k = 3$ and $g >3$. 

We next focus on when $(k,g) = (3,3)$. By Lemma \ref{lemma_notnullcyl}, there is a surface in $K \subset \Omega^3\mathcal{M}_{3}(\mu)^{\text{prim}}$ with a Euclidean cylinder which is not null-homologous. Let $\hat{C}$ be a cylinder on the holonomy cover which projects to this cylinder. Let $\hat{\alpha}$ be a core curve of $\hat{C}$. As in the proof of Lemma \ref{lemma_simplecylinder}, we can shear and then collapse all cylinders $\mathcal{N}$-parallel to $\hat{C}$ so that it is the only cylinder in its $\mathcal{N}$-parallel, and moreover $\mathcal{M}$-parallel, equivalence class on a surface $(\hat{X}',\hat{\omega}')$ in $\mathcal{N}$. Suppose $\tau'$ is the $k$-cyclic automorphism on $\hat{X}'$. It follows by symmetry that $(\tau'^*)^j(\hat{\alpha}^*)$ is the lone cylinder in its $\mathcal{M}$-parallel equivalence class. By the Cylinder Deformation Theorem, $\hat{\alpha}^*,\tau'^*(\hat{\alpha}^*),...,(\tau'^*)^{k-1}(\hat{\alpha}^*)$ are all contained in $p(T_{(\hat{X}',\hat{\omega}')}\mathcal{M})$ and so is $$v := \hat{\alpha}^*+\tau'^*(\hat{\alpha}^*)+...+(\tau'^*)^{k-1}(\hat{\alpha}^*).$$ Because $\pi(\hat{\alpha})$ is not null-homologous, $v$ is non-trivial by the isomorphism $H^1(\hat{X}';\mathbb{C})_1 \cong  H_1(X';\mathbb{C})$ where $X' = \hat{X}'/\tau'.$ Therefore, $v$ is a non-trivial element of $H^1(\hat{X}';\mathbb{C})_1 \cap p(T_{(\hat{X}',\hat{\omega}')}\mathcal{M})$ and $$\text{rank}(\mathcal{M}) \geq 1 + \frac{N + N}{2} \geq 1 + \frac{2(3) -2 + 2(3) -2}{2} = 5.$$ The deficit from full rank is at most $g-1 = 3 - 1 = 2$, so we achieve high rank.

Finally, we consider the case when $g=1$. Because the deficit from full rank is at most $g$, here it is at most one. Therefore, we achieve high rank for all $k>2$ prime if $N \geq 2$ because $$\sum_{\ell=1}^{k} \text{dim}_{\mathbb{C}}H^1(\hat{X};\mathbb{C})_{\zeta^\ell} = (k-1)N.$$ When $g =1$, Lemma \ref{lemma_projection} implies $N = n - \text{card}\{m_1,...,m_n \cap k\mathbb{Z}\}$. Because $\mu$ does not contain entries less than or equal to $-k$, $N$ is at least the number of poles, denoted $P$. $P$ must be non-zero because the sum of the entries of $\mu$ is zero and $\mu$ must be non-empty for the stratum to have primitive $k$-differentials. If $P=1$, then the positive entries of $\mu$ must sum to a non-integer multiple of $k$. Therefore, there is at least one zero whose order is a non-integer multiple of $k$ and $N \geq 2$. When $P \geq 2$, then $N \geq 2$ and we are done.
\end{proof}

Given two subvarieties $\mathcal{M}'$ and $\mathcal{M}''$ inside a stratum, we say they are the same \textit{up to marked points} if they project to the same subvariety under $\mathcal{F}$.
Because $\mathcal{M}$ is high rank, $\mathcal{M}$ is either a component or a locus of holonomy covers of a stratum of quadratic differentials up to marked points by Theorem \ref{theorem_highrank}. 

Suppose that $(Y,\eta)$ is a $(2k)$-differential. Similarly constructed as its holonomy cover, there is a canonical intermediate $2$-cyclic cover which is a $k$-differential $(X,\xi)$ such that the projection map $\pi_2:X \to Y$ satisfies $\pi_2^*\eta = \xi^2$. The holonomy cover $(\hat{X},\hat{\omega})$ of $(Y,\eta)$ is the holonomy cover of $(X,\xi)$ up to marked points. Moreover, we obtain the following commutative diagram by the universal property of canonical covers. 
\[ \begin{tikzcd}
(\hat{X},\hat{\omega}) \arrow{r} \arrow[swap]{d} & (\hat{Y},\hat{\eta}) \arrow{d}{\pi_k} \\%
(X,\xi) \arrow{r}{\pi_2}& (Y,\eta)
\end{tikzcd}
\]  
The canonical intermediate $k$-cyclic cover $(\hat{Y}, \hat{\eta})$ is a quadratic differential whose projection $\pi_k$ to $Y$ satisfies $\pi_k^*\eta = \hat{\eta}^k$.
In fact, these intermediate covers exist for all $k'$-differentials, rather than just $(2k)$-differentials, where $k' = dk''$ for any $d,k'' \in \mathbb{N}$. See \cite[proof of Proposition 5.5]{CG22} for more details. Following \cite[Lemma 3.15 (d)]{EV92}, a singularity $x$ of order $m$ on the $k'$-differential $(X',\xi')$ has $\text{gcd}(m,d)$ pre-images on the canonical intermediate $d$-cyclic cover. Consequently, the ramification index at a pre-image $\hat{x}$ of $x$ on the $d$-cyclic cover is $d/\text{gcd}(m,d)$, and we compute that the order $\hat{m}$ of $\hat{x}$ is \begin{equation} \label{equation_RHintermediate}
\hat{m} = \frac{m+k}{\text{\text{\text{gcd}}}(m,d)} - \frac{k}{d}.
\end{equation}

A translation surface $(S,\omega)$ is a  \textit{translation cover} if there is a translation surface $(Y,\sigma)$ of lower genus and branched covering $f:S \to Y$ such that $f^*\sigma = \omega$. A translation surface is \textit{minimal} if it is not a translation cover. Similarly, $(S,\omega)$ is a \textit{half-translation cover} if there is a half-translation surface $(W,q)$ of lower genus and branched covering $f:S \to W$ such that $f^*q = \omega^2$. 

\begin{lemma} \label{lemma_minimalcover}
Almost every surface in $\mathcal{N}$ is minimal.
\end{lemma}
\begin{proof}
By Lemma \ref{lemma_NdenseinM}, the orbit closure of almost every surface in $\mathcal{N}$ is $\mathcal{M}$. Say $(\hat{X},\hat{\omega})$ is any surface in $\mathcal{N}$ with a dense orbit in $\mathcal{M}$, and suppose it is not minimal. A locus of translation covers is an affine invariant subvariety, so $\overline{GL^+(2,\mathbb{R})(\hat{X},\hat{\omega})}$ consist entirely of covers of translation covers. Since $\mathcal{M}$ is high rank by Lemma \ref{lemma_Mhighrank}, \cite[Lemma 2.1]{AW23} implies $\overline{GL^+(2,\mathbb{R})(\hat{X},\hat{\omega})} \cap \mathcal{M}$ is a proper affine invariant subvariety inside $\mathcal{M}$ which contradicts our assumption on $(\hat{X},\hat{\omega})$.
\end{proof}

Lemma \ref{lemma_uniquecover} will imply when $K$ is hyperelliptic and covers the genus zero surfaces in the ambient stratum of $(Y,\eta)$ in the diagram above, $\mathcal{M}$ is a full locus of holonomy covers of surfaces in the stratum of $(\hat{Y},\hat{\eta})$ up to marked points. Moreover, $\mathcal{M}$ is not contained in a smaller locus of covers of a different stratum of quadratic differentials. 

\begin{lemma} \label{lemma_uniquecover}
Almost every surface in $\mathcal{N}$ is a degree two half-translation cover of at most one half-translation surface.
\end{lemma}
\begin{proof}
By Lemma \ref{lemma_minimalcover}, almost every surface inside $\mathcal{N}$ is minimal. Hence, for generic surfaces in $\mathcal{N}$, \cite[Lemma 3.3]{AW21} implies they are degree two covers of at most one half-translation surface.
\end{proof}

\begin{lemma} \label{lemma_involutiondescends}
Suppose $\mathcal{N}$ consists entirely of holonomy covers of a half-translation surfaces up to marked points. Then a component $K \subset \Omega^k\mathcal{M}_g(\mu)^{\emph{prim}}$ consists entirely of canonical intermediate $2$-cyclic covers of a stratum of $(2k)$-differentials.
\end{lemma}
\begin{proof}
First we claim the holonomy involution $J$ descends to an involution $j$ on $(X,\xi)$ such that $j^*\xi = -\xi$. This is equivalent to showing $\tau$ and $J$ commute, i.e. $J = \tau J \tau^{-1}$. Observe $J=J^{-1}$ and both $J^{-1}$ and $\tau J \tau^{-1}$ are involutions which $\xi$ is $(-1)$-invariant of. If $T := \tau J \tau^{-1}J^{-1}$ is not the identity, then since $T^*\xi = \xi$ and the abelian differential descends to the quotient, $(\hat{X},\hat{\omega})/T$ is a translation surface of smaller genus. This contradicts that $(\hat{X},\hat{\omega})$ is almost always minimal (Lemma \ref{lemma_minimalcover}).

Because the claim is true, we can consider the quotient $(X,\xi)/j$ which is a $(2k)$-differential whose  canonical intermediate $2$-cyclic cover is $(X,\xi)$.
\end{proof}

The following will imply $K$, unless a hyperelliptic component, cannot be $2$-cyclic covers of $(2k)$-differentials. This along with the previous Lemma implies then $\mathcal{N}$ cannot live in a locus of holonomy covers of a stratum of quadratic differentials up to marked points.

\begin{lemma} \label{lemma_notdoublecovers}
A component $K \subset \Omega^k\mathcal{M}_g(\mu)^{\emph{prim}}$ cannot consist entirely of canonical intermediate $2$-cyclic covers of surfaces in a stratum of positive genus $(2k)$-differentials.
\end{lemma}
\begin{proof}
Suppose otherwise and that every surface is such a cover of a genus $h$ surface in the stratum $\Omega^{2k}\mathcal{M}_{h}(\nu)$. Then, the dimension of the corresponding component of $\Omega^{2k}\mathcal{M}_{h}(\nu)$ must be equal to the dimension of $K$. Because $(\hat{X},\hat{\omega})$ is connected, this component in $\Omega^{2k}\mathcal{M}_{h}(\nu)$ must consist of primitive $(2k)$-differentials, and we obtain \begin{equation} \label{equation_dimension}
2g+n-2 =2h+ m -2
\end{equation} where $m$ is the number of singularities of the stratum of $(2k)$-differentials. Let $m_1$ and $m_2$ be the number of even and odd order entries of $\nu$ respectively. The covering map from a surface in $K$ to a surface in $\Omega^{2k}\mathcal{M}_{h}(\nu)$ is only branched over singularities of odd order, and the ramification index at the pre-image is $2$. By the Riemann-Hurwitz formula, $$2g -2 = 2(2h-2) + m_2.$$ Therefore since $n = 2m_1 + m_2$, Equation (\ref{equation_dimension}) becomes $$2(2h-2) + m_2 + (2m_1 + m_2) = 2h+ m_1 + m_2 -2$$ which simplifies to $$2h + m_1 + m_2 = 2.$$ Because $\nu$ is empty only in the stratum $\Omega^{2k}\mathcal{M}_1(\emptyset)$ which parameterizes $(2k)$-th powers of abelian differentials, $\nu$ here is non-empty and $m_1 + m_2 >0$. When $h > 0$, this equality does not hold and we have a contradiction.
\end{proof}

Now we can show that $\mathcal{F(M)}$ is a component or a hyperelliptic locus and are ready to re-introduce marked points. A marked point $y$ on a surface in $\mathcal{M}$ is said to be $\mathcal{M}$\textit{-free} if $\mathcal{M}$ contains all surfaces obtained by moving $y$ while fixing the rest of the surface. 
   
\begin{proof}[Proof of Theorem \ref{theorem_holonomyorbit}]
By Lemma \ref{lemma_Marithmetic}, $\mathcal{F(M)}$ is always arithmetic when $g>2$. Together, Lemma \ref{lemma_Mhighrank} and Theorem \ref{theorem_highrank} imply $\mathcal{F}(\mathcal{M})$ is a component of a stratum or an unmarked locus of holonomy covers of surfaces in a stratum of half-translation surfaces. Lemma \ref{lemma_involutiondescends} implies if $\mathcal{F}(\mathcal{M})$ is the latter, then $K$ covers a stratum of $(2k)$-differentials. Furthermore, Lemma \ref{lemma_notdoublecovers} implies when $K$ is non-hyperelliptic, this cannot happen and $\mathcal{F}(\mathcal{M})$ is necessarily a non-hyperelliptic component of a stratum. By the main result of \cite{Api20}, all marked points on surfaces in $\mathcal{M}$ are $\mathcal{M}$-free. Thus, $\mathcal{M}$ is also a component of a stratum (with possibly marked points).
 
Consider the case where $K$ is hyperelliptic. For every surface $(X,\xi) \in K$ with a hyperelliptic involution $\iota$, we have the following commutative diagram \[ \begin{tikzcd}
(\hat{X},\hat{\omega}) \arrow{r} \arrow[swap]{d} & (\hat{Y},\hat{\eta}) \arrow{d} \\%
(X,\xi) \arrow{r}{/\iota}& (Y,\eta)
\end{tikzcd}
\] where $(Y,\eta) = (X,\xi)/\iota$ and $(\hat{Y},\hat{\eta})$ is its canonical intermediate $k$-cyclic cover and $\hat{\eta}$ thus a quadratic differential. Using Equation (\ref{equation_RHintermediate}) and the number theoretic conditions given in \cite[Theorem 1.1]{CG22}, we compute that $(\hat{Y},\hat{\eta})$ lives in the stratum
\begin{enumerate}
        \item $\Omega^2 \mathcal{M}_0(2m_1+k-2, 2m_2 + k -2, -1^{2gk})$ when $K$ is the hyperelliptic component of $\Omega^k \mathcal{M}_g(2m_1, 2m_2)$,
         \item $\Omega^2 \mathcal{M}_0(2m+k-2, 2\ell + 2k -2, -1^{2gk+k})$ when $K$ is the hyperelliptic component of $\Omega^k \mathcal{M}_g(2m, \ell, \ell)$,
         \item and $\Omega^2 \mathcal{M}_0(2\ell_1+2k-2, 2\ell_2 + 2k -2, -1^{2gk+2k})$ when $K$ is the hyperelliptic component of $\Omega^k \mathcal{M}_g(\ell_1, \ell_1, \ell_2, \ell_2)$.
            \end{enumerate}
Lemma \ref{lemma_uniquecover} implies $\mathcal{F}(\mathcal{M})$ must be the unmarked hyperelliptic locus over the stratum $\mathcal{Q}$ of $(\hat{Y},\hat{\eta})$. By the commutivity of the diagram and Equation (\ref{equation_RHintermediate}), the pre-images of poles on $(\hat{Y},\hat{\eta})$ are also the pre-images of regular Weierstrauss points on $(X,\xi)$ which we do not mark. Hence, all the marked points on $(\hat{X},\hat{\omega})$ are the pre-images of (regular) marked points on $(\hat{Y},\hat{\eta})$. Therefore, the marked points on surfaces in $\mathcal{N}$ must come in pairs interchanged by the holonomy involution. We see from the possible partitions of $\mathcal{Q}$ that at most one pair of points interchanged by the holonomy involution can be marked on surfaces in $\mathcal{N}$. 

If $\mathcal{M}$ is a proper subvariety of the full hyperelliptic locus $\tilde{\mathcal{Q}}$ over the stratum $\mathcal{Q}$, then points in this pair are $\mathcal{F}(\tilde{\mathcal{Q}})$-periodic points, i.e. the dimension of $\mathcal{F}(\tilde{\mathcal{Q}})$ after marking a point of the pair is the dimension of $\mathcal{F}(\tilde{\mathcal{Q}})$. By \cite[Theorem 1.4]{AW21}, there are no such points outside of Weierstrass points. Hence, $\mathcal{M}$ is the full locus $\tilde{\mathcal{Q}}$.
\end{proof}
\subsection{Low genus cases} \label{subsection_LGC}
Concerning the classification of $\mathcal{M}$, there are partial results in low genus. Recall when $(k,g) = (3,2)$, the dimension of $\mathcal{N}$ is also not always high enough to deduce high rank after Lemma \ref{lemma_eigenspacefullrank}. Hence, this case is omitted from the following Theorem. 

\begin{theorem} \label{theorem_LGC}
Suppose that $k>2$ is prime and $(k,g) \neq (3,2)$, and let $K$ be a component of $\Omega^k\mathcal{M}_g (\mu)^{\emph{prim}}$. When $g \leq 2$, almost every $(X,\xi) \in K$ lifts to a surface $(\hat{X},\hat{\omega}) \in \mathcal{H}_{\hat{g}}(\hat{\mu})$ whose $GL^+(2,\mathbb{R})$-orbit closure is either
    \begin{enumerate}
         \item a connected component of $\mathcal{H}_{\hat{g}}(\hat{\mu})$, 
        \item a hyperelliptic locus classified in Theorem \ref{theorem_holonomyorbit} (i),
        \item or non-arithmetic.
    \end{enumerate}
If $g =1$ and the positive entries of $\mu$ sum to be greater than $k$, then almost every $(X,\xi) \in K$ lifts to a surface $(\hat{X},\hat{\omega}) \in \mathcal{H}_{\hat{g}}(\hat{\mu})$ whose $GL^+(2,\mathbb{R})$-orbit closure is either
    \begin{enumerate}
         \item a connected component of $\mathcal{H}_{\hat{g}}(\hat{\mu})$
        \item or a hyperelliptic locus classified in Theorem \ref{theorem_holonomyorbit} (i).
    \end{enumerate} 
\end{theorem}

Recall it is unknown whether the orbit closure of generic holonomy covers can be non-arithmetic in genus one and two, and non-arithmeticity does in fact happen in genus zero. 

In genus one, only when $\mu$ is empty can $\Omega^k\mathcal{M}_{1}(\mu)$ have (and only will have) differentials which are $k$-th powers of abelian differentials (since we do not consider $\mu$ to have higher order poles). Therefore, we can assume $\mu \neq \emptyset$ and take any component $K \subset \Omega^k\mathcal{M}_1(\mu)$ to parameterize primitive $k$-differentials. For any $k$, components of $\Omega^k\mathcal{M}_1(m_1,...,m_n)$ are classified by an invariant called the rotation number. Formally, it is defined as $$\text{rot}(X,\xi) := \text{\text{\text{gcd}}}(\text{Ind}(\alpha),\text{Ind}(\beta), m_1,...,m_n)$$ where $\alpha$ and $\beta$ are curves whose homology classes form a symplectic basis for $H^1(X;\mathbb{Z})$ and $\text{Ind}(\_)$ is the index of a curve. See \cite[Section 3.4]{CG22}.

In \cite[Theorem 3.12]{CG22}, it was proved that for any positive divisor $d$ of $\text{\text{\text{gcd}}}(m_1,...,m_n)$, there is a unique component of $\Omega^k\mathcal{M}_1(m_1,...,m_n)$ which realizes $d$ as its rotation number. The proof of Lemma \ref{lemma_genus1cyl} follows similarly to that of \cite[Theorem 3.12]{CG22} which includes higher order poles.

\begin{lemma} \label{lemma_genus1cyl}
Suppose that $g=1$ and the positive entries of $\mu$ sum up to be greater than $k$. Then, there is a surface in any component $K\subset \Omega^k\mathcal{M}_{1}(\mu)$ which has a Euclidean cylinder. In particular, $\mathcal{F(M)}$ is arithmetic when $k$ is prime.
\end{lemma}
\begin{proof}
Suppose that $m_1,...,m_r$ are the orders (including multiplicities) of all the zeros and $m_{r+1},...,m_n$ of all the poles of $k$-differentials in $\Omega^k\mathcal{M}_{1}(\mu)$. Consider a connected component $K \subset \Omega^k\mathcal{M}_{1}(\mu)$ whose rotation number is $d$ and set $m = m_1+...+m_r$. By \cite[Proposition 3.7]{CG22}, we can perform the bubbling a handle operation on a surface in the stratum $\Omega^k\mathcal{M}_0(m-2k,m_{r+1},...,m_n)$ to obtain a genus one surface $(X',\xi')$ with rotation number $d$ in $\Omega^k\mathcal{M}_{1}(m,m_{r+1},...,m_n)$. Because the positive entries of $\mu$ sum up to be greater than $k$, $m$ satisfies $m -2k >-k$ (which is required to bubble a handle at that singularity). The genus one surface $(X',\xi')$ acquires a Euclidean cylinder from smoothing out the horizontal node. 

Call the bubbled cylinder core cure $\alpha$ and let $\beta$ be the curve that runs through $\alpha$ and turns around the unique zero. The pair $(\alpha,\beta)$ forms a symplectic basis for $H_1(X';\mathbb{Z})$. We then break up the zero of order $m$ into $r$ zeros of orders $m_1,...,m_r$ while using Lemma \ref{lemma_cylpersists} to preserve $\alpha$ as a core curve. Moreover, the indices of $\alpha$ and $\beta$ remain unchanged, and hence the ambient component is of rotation number $d$. Therefore, we land into the component $K$. In particular, Lemma \ref{lemma_simplecylinder} implies $\mathcal{F(M)}$ is arithmetic.
\end{proof}

\begin{proof}[Proof of Theorem \ref{theorem_LGC}]
One can check after assuming $\mathcal{F(M)}$ is arithmetic, the proof follows exactly as the proof of Theorem \ref{theorem_holonomyorbit}. Furthermore, when $g=1$ and the sum of the zeros are greater than $k$, Lemma \ref{lemma_genus1cyl} implies $\mathcal{F(M)}$ is arithmetic. 
\end{proof}

Because it remains an open question whether $\mathcal{F(M)}$ is non-arithmetic when $g\leq2$, we cannot generically determine the weak asymptotics of counting functions on low genus surfaces in the upcoming section. 

\section{Asymptotics of counting functions} \label{theorems}

In this section, we prove Theorem \ref{theorem_main} and talk about Siegel-Veech constants across different components of $\Omega^k\mathcal{M}_g(\mu)^\text{prim}$. Theorem \ref{theorem_main} follows quickly from Theorems \ref{weak} and \ref{theorem_holonomyorbit}.

\begin{proof} [Proof of Theorem \ref{theorem_main}] Theorem \ref{theorem_holonomyorbit} says almost every $M \in K$ has a holonomy cover $\hat{M}$ whose $GL^+(2,\mathbb{R})$-orbit closure is the ambient component $\hat{K}$ in $\mathcal{H}_{\hat{g}}(\hat{\mu})$ when $K$ is non-hyperelliptic or the ambient hyperelliptic locus when $K$ is hyperelliptic. Theorem \ref{weak} implies that the weak asymptotics of $N_{cyl}(\hat{M},L)$ and $N_{sc}(\hat{M},L)$ are given by $$\lim_{L \to \infty}\frac{1}{L}\int^{L}_{0} N_{cyl}(\hat{M},e^t)e^{-2t}dt = \frac{c \cdot \pi}{\text{Area}(\hat{M})}$$ $$\lim_{L \to \infty}\frac{1}{L}\int^{L}_{0} N_{sc}(\hat{M},e^t)e^{-2t}dt = \frac{s \cdot \pi}{\text{Area}(\hat{M})}$$ where the constants $c$ and $s$ depend on the $GL^+(2,\mathbb{R})$-orbit closure of $\hat{M}$.  Therefore, we can almost always take $c$ and $s$ to be $\hat{c}_{cyl}$ and $\hat{c}_{sc}$ respectively which are the Siegel-Veech constants associated to $\hat{K}$ or the ambient hyperelliptic locus therein. By (\ref{equation_count}) and (\ref{equation_area}), we can replace $N_{cyl}(\hat{M},L)$ and $N_{sc}(\hat{M},L)$ with $k \cdot N_{cyl}(M,L)$ and $k \cdot N_{sc}(M,L)$ respectively and $\text{Area}(\hat{M})$ with $k \cdot \text{Area}(M)$ in the equations above. This immediately yields $$N_{cyl}(M,L) `` \sim ” \frac{\hat{c}_{cyl} \cdot \pi L^2}{k^2 \cdot \text{Area}(M)} \hspace{35pt} N_{sc}(M,L) `` \sim ” \frac{\hat{c}_{sc} \cdot \pi L^2}{k^2 \cdot \text{Area}(M)}.$$
\end{proof}

\subsection{Computing Siegel-Veech constants for hyperelliptic components} 
There are many strata which contain holonomy covers of non-hyperelliptic prime-order $k$-differentials. The Siegel-Veech constants for $cyl$ or $sc$ cannot be nicely formulated for arbitrary components of strata of translation surfaces. To compute them, one would first consider every possible configuration of cylinders (resp. saddle connections) that would appear on a surface in that stratum. Then, one would compute the Siegel-Veech constants associated to these configurations using the derived formulas and techniques in \cite{EMZ03}. Afterwards, we sum up these Siegel-Veech constants to obtain the Siegel-Veech constant for $cyl$ (resp. $sc$). 

When $K$ is hyperelliptic however, the ambient hyperelliptic locus of the holonomy covers $\mathcal{N}$ in Theorem \ref{theorem_main} are double covers of a stratum  $\Omega^2\mathcal{M}_0(n_1,n_2,-1^{n_1+n_2 +4})$ for some $n_1,n_2 \geq 0$ (see Theorem \ref{theorem_holonomyorbit}). Only for the hyperelliptic components of the strata $\Omega^k\mathcal{M}_g(2m,-k+1,-k+1)$ and $\Omega^k\mathcal{M}_g(\ell,\ell,-k+1,-k+1)$ does some $n_i =0$. Using the simplicity of this stratum, Siegel-Veech formulas from Athreya-Eskin-Zorich \cite{AEZ16}, and furthermore Apisa \cite{Api21}, the Siegel-Veech constant for $cyl$ is formulated for general $n_1,n_2 > 0$. We now summarize \cite[Section 8]{Api21}.

A cylinder on a half-translation surface is called a \textit{simple cylinder} if each of its boundary components is a saddle connection. A cylinder is called an \textit{envelope} if one of its boundary components is a saddle connection of multiplicity two and the other a single saddle connection. Let $c_{simp}$ and $c_{env}$ be the Siegel-Veech constants for the configuration of any simple cylinder and any envelope respectively. Then, Apisa \cite[Corollary 8.3]{Api21} using the Siegel-Veech constants for Configurations III and IV of Athreya-Eskin-Zorich \cite{AEZ16} showed that for the stratum $\Omega^2\mathcal{M}_0(n_1,n_2,-1^{n_1+n_2 +4})$ with $n_1,n_2 >0$, $$ c_{simp} = \frac{1}{2\pi^2}\binom{n_1+n_2+4}{2}\frac{2}{(n_1+2)(n_2+2)}$$
$$c_{env} = \frac{1}{2\pi^2} \binom{n_1+n_2+4}{2}.$$ 

We here explain the proof of \cite[Theorem 8.4]{Api21}. On a full measure set in a genus zero stratum of quadratic differentials other than $\Omega^2\mathcal{M}_0(-1^4)$, every cylinder is either a simple cylinder or an envelope (see \cite{MZ08} or \cite[Section 4.1]{AW24}). Let $(S,\omega)$ be a hyperelliptic surface which is a double cover of a surface in this full measure set in $\Omega^2\mathcal{M}_0(n_1,n_2,-1^{n_1+n_2 +4})$. Suppose $\iota$ is its hyperelliptic involution. Since the pre-images of poles are unmarked on $(S,\omega)$, simple cylinders on $(S,\omega)/\iota$ have two cylinders in the pre-image on $(S,\omega)$ and envelopes have one. Hence, $$\hat{c}_{cyl} = 2 c_{simp} +  c_{env}$$ where $\hat{c}_{cyl}$ is the Siegel-Veech constant counting cylinders for the hyperelliptic locus covering $\Omega^2\mathcal{M}_0(n_1,n_2,-1^{n_1+n_2 +4})$. We then plug this constant into Theorem \ref{theorem_main}. 

Using the classification of the hyperelliptic locus in Theorem \ref{theorem_holonomyorbit} and \cite[Theorem 8.4]{Api21}, we have for a generic surface $M \in K$ without a pole of order $k-1$,

\vspace*{5px}
$\begin{aligned}
    N_{cyl}(M,L)``\sim"
\end{aligned}$
$$\frac{1}{2\pi^2} \binom{2m_1 +2m_2 +2k}{2} \left( 1 + \frac{4}{(2m_1 +k)(2m_2 +k)}\right) \frac{2\pi L^2}{k^2 \cdot \mathrm{Area}(M)}$$ when $K \subset \Omega^k\mathcal{M}_g(2m_1,2m_2)$,
$$\frac{1}{2\pi^2} \binom{2m +2\ell +3k}{2} \left( 1 + \frac{4}{(2m +k)(2\ell +2k)}\right) \frac{2\pi L^2}{k^2 \cdot \mathrm{Area}(M)}$$
when $K \subset \Omega^k\mathcal{M}_g(2m,\ell,\ell)$, and
$$\frac{1}{2\pi^2} \binom{2\ell_1 +2\ell_2 +4k}{2} \left( 1 + \frac{4}{(2\ell_1 +2k)(2\ell_2 +2k)}\right) \frac{2\pi L^2}{k^2 \cdot \mathrm{Area}(M)}$$ when $K \subset \Omega^k\mathcal{M}_g(\ell_1,\ell_1,\ell_2,\ell_2)$.

In contrast to $\hat{c}_{cyl}$, there is not a nice general formula for $\hat{c}_{sc}$ even for the hyperelliptic loci we consider. However, the method for computing $\hat{c}_{sc}$ (or $\hat{c}_{cyl}$ with a pole of order $k-1$) is the same as above in that we categorize the configurations of saddle connections (resp. cylinders) on surfaces in $\Omega^2\mathcal{M}_0(n_1,n_2,-1^{n_1+n_2 +4})$ based on the number of saddle connections (resp. cylinders) in their pre-image on the hyperelliptic surface. The pre-image of a saddle connection on a surface in $\Omega^2\mathcal{M}_0(n_1,n_2,-1^{n_1+n_2 +4})$ has zero saddle connections when it connects two poles, one saddle connection when it connects a pole to a zero, or two saddle connections when it connects two (not necessarily distinct) zeros. One then uses \cite{AEZ16} to obtain Siegel-Veech constants for each of the three categories and takes their sum weighting each term accordingly by $0,1,$ or $2$.

\subsection{Parity of non-hyperelliptic components} At large, we see that constants $\hat{c}_{cyl}$ and $\hat{c}_{sc}$ depend on the ambient component of $\mathcal{N}$ in $\mathcal{H}_{\hat{g}}(\hat{\mu})$, which depends on the component $K$. In \cite{KZ03}, Kontsevich-Zorich classified components of strata of translation surfaces by hyperellipticity and parity of spin structure. Given a symplectic basis $(\alpha_1,...,\alpha_g,\beta_1,...,\beta_g)$ of $H_1(X;\mathbb{Z}/2)$, the parity of a translation surface $(S,\omega)$ is defined as the parity of the Arf-invariant $$\Phi(\omega) :=  \sum_{i=1}^g(\text{Ind}(\alpha_i) + 1)(\text{Ind}(\beta_i) + 1) \text{ mod } 2$$ where $\text{Ind}$ is with respect to $\omega$. Parity is an invariant of a component of a stratum. Two components of a stratum can have different parity type only when the singularities are all of even order. The parity of a component of a stratum of $(1/k)$-translation surfaces is defined as the parity of its holonomy covers.

There is not a complete classification of components of strata $\Omega^k\mathcal{M}_g(\mu)$, but Chen-Gendron \cite{CG22} partially classified components based on hyperellipticity and parity. When $k$ is even, they show the parity is an invariant of the locus of primitive $k$-differentials. When $k$ is odd, they show strata may have two components of different parity if they lift to strata of only even singularities, and the locus of differentials with the same parity may be disconnected. The 2-adic valuation of $k$ is the highest exponent $v_2(k)$ such that $2^{v_2(k)}$ divides $k$. A small computation shows $\mu$ has only even entries if and only if the 2-adic valuation of every entry of $\mu$ is not equal to $v_2(k)$. Otherwise, all primitive non-hyperelliptic components of a stratum will share the same asymptotic in Theorem \ref{theorem_main}.

\bibliography{references}{}

\newcommand{\etalchar}[1]{$^{#1}$}
\providecommand{\bysame}{\leavevmode\hbox to3em{\hrulefill}\thinspace}
\providecommand{\MR}{\relax\ifhmode\unskip\space\fi MR }
\providecommand{\MRhref}[2]{%
  \href{http://www.ams.org/mathscinet-getitem?mr=#1}{#2}
}
\providecommand{\href}[2]{#2}
\begin{thebibliography}{EMWM06}

\bibitem[AAH22]{AAH22}
Jayadev~S. Athreya, David Aulicino, and W.~Patrick Hooper, \emph{Platonic solids and high genus covers of lattice surfaces}, Exp. Math. \textbf{31} (2022), no.~3, 847--877, With an appendix by Anja Randecker. \MR{4477409}

\bibitem[AEZ16]{AEZ16}
Jayadev~S. Athreya, Alex Eskin, and Anton Zorich, \emph{Right-angled billiards and volumes of moduli spaces of quadratic differentials on {$\Bbb C\rm P^1$}}, Ann. Sci. \'Ec. Norm. Sup\'er. (4) \textbf{49} (2016), no.~6, 1311--1386, With an appendix by Jon Chaika. \MR{3592359}

\bibitem[Api20]{Api20}
Paul Apisa, \emph{{${\rm GL}_2\Bbb{R}$}-invariant measures in marked strata: generic marked points, {E}arle-{K}ra for strata and illumination}, Geom. Topol. \textbf{24} (2020), no.~1, 373--408. \MR{4080485}

\bibitem[Api21]{Api21}
\bysame, \emph{Billiards in right triangles and orbit closures in genus zero strata}.

\bibitem[AW21]{AW21}
Paul Apisa and Alex Wright, \emph{Marked points on translation surfaces}, Geom. Topol. \textbf{25} (2021), no.~6, 2913--2961. \MR{4347308}

\bibitem[AW23]{AW23}
P.~Apisa and A.~Wright, \emph{High rank invariant subvarieties}, Ann. of Math. (2) \textbf{198} (2023), no.~2, 657--726. \MR{4635302}

\bibitem[AW24]{AW24}
Paul Apisa and Alex Wright, \emph{Reconstructing orbit closures from their boundaries}, Mem. Amer. Math. Soc. \textbf{298} (2024), no.~1487, v+141. \MR{4772262}

\bibitem[BCG{\etalchar{+}}19]{BCGGM19}
Matt Bainbridge, Dawei Chen, Quentin Gendron, Samuel Grushevsky, and Martin M\"oller, \emph{Strata of {$k$}-differentials}, Algebr. Geom. \textbf{6} (2019), no.~2, 196--233. \MR{3914751}

\bibitem[BCG{\etalchar{+}}24]{BCGGM24}
Matt Bainbridge, Dawei Chen, Quentin Gendron, Samuel Grushevsky, and Martin Möller, \emph{The moduli space of multi-scale differentials}, 2024.

\bibitem[CG22]{CG22}
Dawei Chen and Quentin Gendron, \emph{Towards a classification of connected components of the strata of {$k$}-differentials}, Doc. Math. \textbf{27} (2022), 1031--1100. \MR{4452231}

\bibitem[Che19]{Chen19}
Dawei Chen, \emph{Affine geometry of strata of differentials}, J. Inst. Math. Jussieu \textbf{18} (2019), no.~6, 1331--1340. \MR{4021107}

\bibitem[CMZ24]{CMM24}
Matteo Costantini, Martin M\"oller, and Jonathan Zachhuber, \emph{The area is a good enough metric}, Ann. Inst. Fourier (Grenoble) \textbf{74} (2024), no.~3, 1017--1059. \MR{4770336}

\bibitem[Doz24]{Doz24}
Benjamin Dozier, \emph{Compactifications of strata of differentials}, 2024.

\bibitem[EM01]{EM01}
Alex Eskin and Howard Masur, \emph{Asymptotic formulas on flat surfaces}, Ergodic Theory Dynam. Systems \textbf{21} (2001), no.~2, 443--478. \MR{1827113}

\bibitem[EMM15]{EMM15}
Alex Eskin, Maryam Mirzakhani, and Amir Mohammadi, \emph{Isolation, equidistribution, and orbit closures for the {${\rm SL}(2,\Bbb R)$} action on moduli space}, Ann. of Math. (2) \textbf{182} (2015), no.~2, 673--721. \MR{3418528}

\bibitem[EMWM06]{EMWM06}
Alex Eskin, Jens Marklof, and Dave Witte~Morris, \emph{Unipotent flows on the space of branched covers of {V}eech surfaces}, Ergodic Theory Dynam. Systems \textbf{26} (2006), no.~1, 129--162. \MR{2201941}

\bibitem[EMZ03]{EMZ03}
Alex Eskin, Howard Masur, and Anton Zorich, \emph{Moduli spaces of abelian differentials: the principal boundary, counting problems, and the {S}iegel-{V}eech constants}, Publ. Math. Inst. Hautes \'Etudes Sci. (2003), no.~97, 61--179. \MR{2010740}

\bibitem[EO01]{EO01}
Alex Eskin and Andrei Okounkov, \emph{Asymptotics of numbers of branched coverings of a torus and volumes of moduli spaces of holomorphic differentials}, Invent. Math. \textbf{145} (2001), no.~1, 59--103. \MR{1839286}

\bibitem[EV92]{EV92}
H\'el\`ene Esnault and Eckart Viehweg, \emph{Lectures on vanishing theorems}, Basel: Birkh\"{a}user Verlag (1992), 164.

\bibitem[Fil16]{Fil16}
Simion Filip, \emph{Splitting mixed {H}odge structures over affine invariant manifolds}, Ann. of Math. (2) \textbf{183} (2016), no.~2, 681--713. \MR{3450485}

\bibitem[FK92]{FK92}
H.~M. Farkas and I.~Kra, \emph{Riemann surfaces}, second ed., Graduate Texts in Mathematics, vol.~71, Springer-Verlag, New York, 1992. \MR{1139765}

\bibitem[Gou15]{Gou15}
Elise Goujard, \emph{Siegel-{V}eech constants for strata of moduli spaces of quadratic differentials}, Geom. Funct. Anal. \textbf{25} (2015), no.~5, 1440--1492. \MR{3426059}

\bibitem[Gou16]{Gou16}
\bysame, \emph{Volumes of strata of moduli spaces of quadratic differentials: getting explicit values}, Ann. Inst. Fourier (Grenoble) \textbf{66} (2016), no.~6, 2203--2251. \MR{3580171}

\bibitem[KZ03]{KZ03}
Maxim Kontsevich and Anton Zorich, \emph{Connected components of the moduli spaces of {A}belian differentials with prescribed singularities}, Invent. Math. \textbf{153} (2003), no.~3, 631--678. \MR{2000471}

\bibitem[MW18]{MW18}
Maryam Mirzakhani and Alex Wright, \emph{Full-rank affine invariant submanifolds}, Duke Math. J. \textbf{167} (2018), no.~1, 1--40. \MR{3743698}

\bibitem[MZ08]{MZ08}
Howard Masur and Anton Zorich, \emph{Multiple saddle connections on flat surfaces and the principal boundary of the moduli spaces of quadratic differentials}, Geom. Funct. Anal. \textbf{18} (2008), no.~3, 919--987. \MR{2439000}

\bibitem[Ngu22]{Ngu22}
Duc-Manh Nguyen, \emph{Volume forms on moduli spaces of {$d$}-differentials}, Geom. Topol. \textbf{26} (2022), no.~7, 3173--3220. \MR{4540904}

\bibitem[Vee89]{Vee89}
W.~A. Veech, \emph{Teichm\"uller curves in moduli space, {E}isenstein series and an application to triangular billiards}, Invent. Math. \textbf{97} (1989), no.~3, 553--583. \MR{1005006}

\bibitem[Wri14]{Wri14}
Alex Wright, \emph{The field of definition of affine invariant submanifolds of the moduli space of abelian differentials}, Geom. Topol. \textbf{18} (2014), no.~3, 1323--1341. \MR{3254934}

\bibitem[Wri15a]{Wri15a}
\bysame, \emph{Cylinder deformations in orbit closures of translation surfaces}, Geom. Topol. \textbf{19} (2015), no.~1, 413--438. \MR{3318755}

\bibitem[Wri15b]{Wri15b}
\bysame, \emph{Translation surfaces and their orbit closures: an introduction for a broad audience}, EMS Surv. Math. Sci. \textbf{2} (2015), no.~1, 63--108. \MR{3354955}

\end{thebibliography}
\bibliographystyle{amsalpha.bst}

\end{document}